\documentclass[11pt]{article}
\usepackage[T1]{fontenc}
\usepackage[utf8]{inputenc}
\usepackage[english]{babel}
\usepackage{microtype}
\usepackage{amsmath,amssymb,amsfonts,amsthm,mathtools}
\usepackage[shortlabels]{enumitem}
\usepackage{geometry}
\usepackage[colorlinks=true,linkcolor=blue,citecolor=blue,urlcolor=blue]{hyperref}
\newtheorem{theorem}{Theorem}[section]
\newtheorem{lemma}[theorem]{Lemma}
\newtheorem{claim}[theorem]{Claim}
\newtheorem{proposition}[theorem]{Proposition}
\newtheorem{corollary}[theorem]{Corollary}
\newtheorem{definition}[theorem]{Definition}

\newtheorem{problem}[theorem]{Problem}
\newtheorem{conjecture}[theorem]{Conjecture}

\newcommand{\eps}{\varepsilon}
\newcommand{\cA}{\mathcal A}

\newcommand{\cC}{\mathcal C}

\newcommand{\cF}{\mathcal F}

\newcommand{\cI}{\mathcal I}
\newcommand{\cM}{\mathcal M}
\newcommand{\cN}{\mathcal N}
\newcommand{\cP}{\mathcal P}
\newcommand{\cQ}{\mathcal Q}

\newcommand{\cT}{\mathcal T}

\newcommand{\floor}[1]{\lfloor #1\rfloor}
\newcommand{\ceil}[1]{\lceil #1\rceil}
\newcommand{\R}{\mathbb R}
\newcommand{\1}{\mathbf 1}
\newcommand{\dd}{\,d}

\usepackage{xcolor}

\DeclareMathOperator{\homg}{hom}
\DeclareMathOperator{\Emb}{Emb}
\DeclareMathOperator{\Aut}{Aut}

\newcommand{\Mbip}{M_{\mathrm{bip}}}
\newcommand{\Nbip}{N_{\mathrm{bip}}}
\title{Finite-Kernel Extremizers in Sparse Extremal Graph Counting}
\author{Jiasheng Zeng\thanks{Email: \href{mailto:jasonzeng@mail.ustc.edu.cn}{jasonzeng@mail.ustc.edu.cn}}}
\date{\today}
\begin{document}
\maketitle

\begin{abstract}
We develop a finite-kernel framework for sparse extremal graph counting. The problems considered here ask for the maximum number of copies or homomorphisms of a fixed graph under sparse edge constraints. In this regime, the leading term need not be governed by a single dense block. Instead, the extremal mass may be supported on several interacting asymptotic scales. Our framework identifies these scales via a finite-dimensional linear program, separates the leading contributions through a finite state decomposition, and synchronizes or realizes them inside a finite kernel.

We apply this framework in three settings. First, we prove the sparse threshold conjecture of Day and Sarkar for graphons. For every fixed graph $H$ without isolated vertices, we prove that
\[
\sup_{t(K_2,W)\le \beta} t(H,W)=\beta^{|V(H)|-\alpha^*(H)}(C_T(H)+o(1))
\]
as $\beta\to0$, where $\alpha^*(H)$ is the fractional independence number of $H$ and $C_T(H)$ is an explicit sharp constant attained by a three-step threshold graphon. Second, we affirmatively answer a question of Blekherman and Patel by showing that, for every graph $H$, whenever $m\to\infty$ and $m=o(n^{3/2})$, threshold graphs asymptotically maximize $\homg(H,G)$ among all graphs with at most $n$ vertices and at most $m$ edges. Third, Gerbner, Nagy, Patk\'os, and Vizer conjectured that, among all bipartite graphs with $n$ vertices and $m$ edges, the quasi-complete bipartite graph asymptotically maximizes the number of copies of every fixed bipartite graph $H$ whenever $m=\omega(n)$ and $m\le n^2/4$. We disprove this conjecture in the subquadratic range and give the correct order of magnitude in terms of $\kappa_H(n,m)$, a finite-kernel scale defined by a finite-dimensional variational problem.
\end{abstract}

\section{Introduction}\label{secintro}

A central topic in extremal graph theory is to determine, under given constraints, the maximum or minimum number of occurrences of a fixed graph $H$. The earliest results trace back to Mantel's theorem \cite{Mantel1907}, which states that if an $n$-vertex graph contains no $K_3$, then the number of edges is at most $\lfloor n^2/4\rfloor$. Tur\'an's theorem extends this result to $K_{r+1}$-free graphs, characterizing the extremal structure that maximizes the number of edges in such graphs \cite{Turan1941}. Erd\H{o}s and Stone \cite{ErdosStone1946}, and Erd\H{o}s and Simonovits \cite{ErdosSimonovits1966} subsequently gave asymptotic solutions for all non-bipartite forbidden graphs. These results themselves maximize the number of $K_2$ while forbidding a fixed graph. In the forbidden-host direction, the generalized Tur\'an problem of Alon and Shikhelman can also be interpreted as maximizing the number of copies of a fixed graph under additional constraints \cite{AlonShikhelman2016}. The subsequent supersaturation problem asks how many copies of the forbidden graph must appear once the number of edges exceeds the Tur\'an threshold. Rademacher--Erd\H{o}s type problems and the Erd\H{o}s--Simonovits supersaturation theory systematized this idea \cite{Erdos1962,ErdosSimonovits1983}. For cliques, Lov\'asz and Simonovits proposed the more precise clique density problem \cite{LovaszSimonovits1976,LovaszSimonovits1983}. Razborov resolved the triangle density problem using flag algebras \cite{Razborov2007}, and Reiher finally proved the clique density theorem \cite{Reiher2016}.

A complementary direction, known as fixed-edge extremal counting, reverses the perspective. Instead of forbidding a graph under a vertex budget, one maximizes the number of copies of a fixed graph $H$ over all graphs with a given number of edges. Alon systematically studied this problem in 1981 and 1986, determining the order of magnitude of the maximum number of copies of a fixed graph over all graphs with a prescribed number of edges \cite{Alon1981,Alon1986}. He further conjectured that the corresponding leading constant exists~\cite{Alon1986}. Very recently, Kuang, Sun, Wang, and Zeng proved this conjecture and characterized the limiting constant by a finite-core variational problem \cite{KuangSunWangZeng2026}. This direction can be viewed as an extension of Kruskal--Katona type problems to general graph counting. For cliques, the optimal structure is given by an initial colex segment or a quasi-clique type structure. For general $H$, the extremal structure may switch between concentrating edges and expanding the vertex set. Later, Friedgut and Kahn extended such fixed-edge counting to hypergraphs and gave sharp orders up to multiplicative constants \cite{FriedgutKahn1998}. Kopparty and Rossman studied related exponent problems from the perspective of the homomorphism domination exponent \cite{KoppartyRossman2011}. Recently, Chao and Yu used the entropy method to study Kruskal--Katona type problems and provided several unified generalizations \cite{ChaoYu2023}; their work on hypergraph joints further connects fixed-edge copy counting with incidence geometry \cite{ChaoYu2024}. Induced counting forms another branch, where the inducibility problem introduced by Pippenger and Golumbic studies the maximum density of a fixed graph as an induced subgraph \cite{PippengerGolumbic1975}. The edge inducibility problem proposed recently by Chao, Cohen Antonir, Li, and Yu can be viewed as a common refinement of the Kruskal--Katona theorem and the inducibility problem \cite{ChaoCohenAntonirLiYu2025}.

Another important direction is to maximize or minimize homomorphism densities in the graphon or dense graph setting with a fixed edge density. Graph limit theory provides a unified language for such problems. A graphon is a symmetric measurable function $W:[0,1]^2\to[0,1]$, and $t(H,W)$ denotes the homomorphism density of $H$ in $W$; in particular, $t(K_2,W)$ is the edge density of $W$. Two graphons $U$ and $W$ are \textit{weakly isomorphic} if $t(F,U)=t(F,W)$ for every finite graph $F$. Lov\'asz and Szegedy established the basic framework of dense graph limits \cite{LovaszSzegedy2006}, and Borgs, Chayes, Lov\'asz, S\'os, and Vesztergombi systematically studied homomorphism densities and graph convergence \cite{BorgsChayesLovaszSosVesztergombi2006,BorgsChayesLovaszSosVesztergombi2008}. Lov\'asz's monograph gives a complete background \cite{Lovasz2012}. In this setting, many extremal problems can be transformed into finite-dimensional or infinite-dimensional optimization problems on graphons. Blekherman and Patel recently proved that, under a fixed edge density, threshold graphs asymptotically maximize homomorphism densities \cite{BlekhermanPatel2024}. In the sparser regime of graph limit theory, Bollob\'as and Riordan, and Borgs, Chayes, Cohn, and Zhao gave different frameworks for sparse convergence \cite{BollobasRiordan2009,BorgsChayesCohnZhao2019}. Day and Sarkar studied Nagy's conjecture \cite{Nagy2017} and related extremal density problems, and proposed several new conjectures about threshold-type extremal graphs \cite{DaySarkar2021}. Specifically, for a graph $H$ with no isolated vertices, define the fractional independence number 
\[ \alpha^*(H)=\max\left\{\sum_{v\in V(H)}x_v\ \middle|\ x_u+x_v\le1\ \text{for every }uv\in E(H),\ x_v\ge0\ \text{for every }v\in V(H)\right\}. \] 
Let $\Phi^*(H)$ be the set of all functions $\phi:V(H)\to\{0,1/2,1\}$ such that $\phi(u)+\phi(v)\le1$ for every edge $uv\in E(H)$ and $\sum_v\phi(v)=\alpha^*(H)$. For $\phi\in\Phi^*(H)$, define 
\[ r_\phi=|\phi^{-1}(0)|,\qquad y_\phi=|\phi^{-1}(1/2)|,\qquad b_\phi=|\phi^{-1}(1)|. \] 
Define also \[ P_H(q)=\sum_{\phi\in\Phi^*(H)} \left(\frac{1-q^2}{2}\right)^{r_\phi}q^{y_\phi}, \qquad C_T(H)=\max_{0\le q\le1}P_H(q). \] For $0\le q\le1$ and $0<\beta<1$, define 
\[ D_{\beta,q}=[0,\sqrt\beta q]\cup[\sqrt{1-\beta(1-q^2)},1], \qquad T_\beta(q)=W_{D_{\beta,q}}, \] 
where $W_D(x,y)=\1_{\max\{x,y\}\in D}$. This graphon $T_\beta(q)$ is the three-step threshold graphon of Day and Sarkar, and its edge density equals $\beta$. Define
\[
\cM_H(\beta)=\sup\{t(H,W)\mid W\text{ is a graphon and }t(K_2,W)\le\beta\}
\]
and let
\[
\cM^T_H(\beta)=\sup_{0\le q\le1}t(H,T_\beta(q)).
\]
Day and Sarkar conjectured that, in the sparse limit, the three-step threshold graphons already give the correct asymptotic extremal value. In the notation of this paper, their sparse conjecture can be stated as follows.

\begin{conjecture}[Day and Sarkar, reformulated from~\cite{DaySarkar2021}]\label{sourceDS}
For every fixed graph $H$ without isolated vertices, $
\cM_H(\beta)= (1+o(1))\cM^T_H(\beta)$ 
as $\beta\rightarrow0.$
\end{conjecture}

Our first main theorem proves this conjecture and identifies the sharp leading term.

\begin{theorem}\label{thmds}
For every fixed graph $H$ without isolated vertices, as $\beta\rightarrow0$ one has
\[
\cM_H(\beta)=\beta^{|V(H)|-\alpha^*(H)}(C_T(H)+o(1))
\]
and
\[
\cM^T_H(\beta)=\beta^{|V(H)|-\alpha^*(H)}(C_T(H)+o(1)).
\]
Consequently $\cM_H(\beta)=(1+o(1))\cM^T_H(\beta)$ as $\beta\rightarrow0$.
\end{theorem}

Next we turn from graphons to finite graphs and fix both the number of vertices and the number of edges. This setting lies between fixed-edge counting and dense graphon optimization. The number of vertices $n$ limits the available space, and the number of edges $m$ limits the available resources. The theorem of Ahlswede and Katona on adjacent pairs of edges already demonstrates a competition between quasi-star and quasi-clique structures \cite{AhlswedeKatona1978}. Cutler and Radcliffe studied the corresponding extremal graphs from the perspective of homomorphism counts \cite{CutlerRadcliffe2011,CutlerRadcliffe2014}, and Reiher and Wagner investigated maximum star densities \cite{ReiherWagner2018}. A finite graph $T$ is a threshold graph if it can be obtained from the empty graph by successively adding an isolated vertex or a dominating vertex. We denote by $\cT$ the class of threshold graphs. We use the standard structure theory of threshold graphs from \cite{ChvatalHammer1977,MahadevPeled1995}, and the graphon limit representation of threshold graphs from \cite{DiaconisHolmesJanson2008}. Blekherman and Patel proved that, for every fixed graph $H$ and every $c\in[0,1]$, the asymptotic maximum of $t(H,G)$ over graphs satisfying $t(K_2,G)\le c$ is attained by threshold graphs. They also proved a sparse finite-graph analogue in the range $m=\omega(n^{3/2})$ and they further asked whether threshold graphs remain asymptotically optimal in the sparser regime of finite graphs~\cite{BlekhermanPatel2024}. For graphs $H$ and $G$, let $\homg(H,G)$ denote the number of homomorphisms from $H$ to $G$, and let $\Emb(H,G)$ denote the number of injective homomorphisms from $H$ to $G$. 
Let
\[\cM(H,n,m)=\max\{\homg(H,G)\mid |V(G)|\le n,\ e(G)\le m\}.\] 
The question of Blekherman and Patel can be stated as follows. 

\begin{problem}[Blekherman and Patel, Question 50~\cite{BlekhermanPatel2024}]\label{sourceBP} Determine whether, for every fixed graph $H$ and every function $m=m(n)$ satisfying $m\to\infty$, $m\le\binom n2$, and $m=o(n^{3/2})$, the maximum $\cM(H,n,m)$ is asymptotically attained by threshold graphs. \end{problem} 

Our second main theorem gives an affirmative answer to Problem \ref{sourceBP}.

\begin{theorem}\label{thmbp}
Let $H$ be a fixed finite simple graph. Let $m=m(n)$ satisfy $m\to\infty$, $m\le\binom n2$, and $m=o(n^{3/2})$. Then
\[
\cM(H,n,m)= (1+o(1))
\max\{\homg(H,T)\mid |V(T)|\le n,\ e(T)\le m,\ T\in\cT\}.
\]
\end{theorem}

The proof of Theorem \ref{thmbp} proceeds through a stronger simultaneous statement for finite weighted families of embedding counts. Let $\cF$ be a finite family of fixed graphs with no isolated vertices and let $a_F\ge0$ for $F\in\cF$. At least one weight is assumed positive.

\begin{theorem}\label{thmweightedthreshold}
Assume $m=m(n)\to\infty$ as $n\to \infty$, $m\le\binom n2$, and $m=o(n^{3/2})$. Then
\[
\max_{\substack{|V(G)|\le n\\ e(G)\le m}}
\sum_{F\in\cF}a_F\Emb(F,G)
= (1+o(1))
\max_{\substack{|V(T)|\le n\\ e(T)\le m\\ T\in\cT}}
\sum_{F\in\cF}a_F\Emb(F,T).
\]
\end{theorem}

We now turn to another fixed-order-and-size problem studied by Gerbner, Nagy, Patk\'os, and Vizer \cite{GNPV2021}. They considered maximizing the number of copies of a fixed graph $H$ in a graph with given order and size, and proposed a natural conjecture in the bipartite host setting, which states that when the host graph is required to be bipartite, the quasi-complete bipartite graph should be asymptotically optimal. For graphs $G$ and $H$, let $N(G,H)$ denote the number of unlabelled copies of $H$ in $G$, let $M(G,H)$ denote the number of labelled embeddings of $H$ in $G$, and let $\Aut(H)$ denote the automorphism group of $H$. Hence $M(G,H)=|\Aut(H)|N(G,H)$. In the bipartite host setting, define $\Nbip(n,m,H)$ and $\Mbip(n,m,H)$ as the maximum of these two quantities over all bipartite graphs $G$ with $|V(G)|=n$ and $e(G)=m$.

Let $B_n^m$ be the quasi-complete bipartite graph. Specifically, take the smallest integer $t$ such that $t(n-t)\ge m$, start from $K_{t,n-t}$, and delete $t(n-t)-m$ edges all incident to one vertex in the part of size $t$. The conjecture of Gerbner, Nagy, Patk\'os, and Vizer \cite{GNPV2021} can be stated as follows.

\begin{conjecture}[Gerbner, Nagy, Patk\'os, and Vizer~\cite{GNPV2021}]\label{sourceGNPV}
For every fixed bipartite graph $H$, if $m=\omega(n)$ and $m\le n^2/4$, then
\[
\Nbip(n,m,H)=(1+o(1))N(B_n^m,H).
\]
\end{conjecture}

Conjecture \ref{sourceGNPV} is false in the subquadratic range. The correct order-level replacement is given by a finite-kernel scale. Define
\[
\kappa_H(n,m)=\max\left\{\prod_{v\in V(H)}a_v\ \middle|\ 1\le a_v\le n\ \text{for every }v\in V(H),\ a_ua_v\le m\ \text{for every }uv\in E(H)\right\}.
\]
The variables $a_v$ are real variables. They describe the continuous finite-kernel scale. Integer vertex classes are obtained only after a fixed scaling and a rounding step.

Let $H_1,\ldots,H_r$ be the connected components of a bipartite graph $H$. For each component $H_i$, fix a bipartition $A_i\cup B_i$ with $|A_i|\le |B_i|$. Define $\sigma(H)=\sum_i|A_i|$ and
\[
\Delta(H)=\max_{1\le i\le r}\max_{X\subseteq A_i}\bigl(|X|-|N_{H_i}(X)|\bigr).
\]
By Hall's theorem, $\Delta(H)=0$ if and only if every chosen smaller colour class $A_i$ is saturated by a matching in $H_i$ \cite[Section 2.1]{Diestel2017}.

\begin{theorem}\label{thmbipmain}
Let $H$ be a fixed bipartite graph with no isolated vertices. The following statements hold.

\begin{enumerate}[(i)]
\item There are constants $c_H,C_H>0$ and $n_H$ such that, for all integers $n,m$ with $n\ge n_H$ and $n\le m\le\floor{n^2/4}$,
\[
c_H\kappa_H(n,m)\le \Mbip(n,m,H)\le C_H\kappa_H(n,m).
\]
Consequently $\Nbip(n,m,H)=\Theta_H(\kappa_H(n,m))$ in the same range.

\item If $m=m(n)$ satisfies $m/n\to\infty$ and $m/n^2\to0$, then $\Delta(H)=0$ implies
\[
\Nbip(n,m,H)=\Theta_H(N(B_n^m,H)).
\]

\item If $m=m(n)$ satisfies $m/n\to\infty$ and $m/n^2\to0$, and if $d=\Delta(H)>0$, then there is a constant $c_H'>0$ such that, for all sufficiently large $n$,
\[
\Nbip(n,m,H)\ge c_H'N(B_n^m,H)\left(\frac{n^2}{m}\right)^d.
\]
In particular $\Nbip(n,m,H)/N(B_n^m,H)\to\infty$.
\end{enumerate}
\end{theorem}

Thus Theorem \ref{thmbipmain} replaces the quasi-complete bipartite conjecture by the finite-kernel order scale $\kappa_H(n,m)$. It also gives the exact order-level criterion for the quasi-complete bipartite graph in the subquadratic range $m=\omega(n)$ and $m=o(n^2)$. In that range, $B_n^m$ has the correct order if $\Delta(H)=0$, and it is smaller than the optimum by a factor tending to infinity if $\Delta(H)>0$. In particular, every double star $D_{a,b}$ with $a,b\ge2$ gives a connected counterexample to Conjecture \ref{sourceGNPV} throughout the range $m=\omega(n)$ and $m=o(n^2)$.

\subsection{The finite-kernel mechanism}

Theorems~\ref{thmds}, \ref{thmbp}, and~\ref{thmbipmain} are proved by a common finite-kernel mechanism. In sparse extremal counting, the leading contribution need not come from a single dense block. Instead, an extremal object may distribute its mass across several asymptotic scales, and the leading term is governed by a finite-dimensional multiscale kernel.

The proof has the same structure in all three applications. First, a linear program identifies the sparse scales that can contribute to the leading order. Second, the equality cases of this program give a finite collection of leading states, while all other states are separated either by a positive exponent gap or by a synchronization argument. Third, the remaining leading states are realized by a bounded kernel.

In the graphon proof of Theorem~\ref{thmds}, the linear program is the fractional independence program for $H$, and its half-integral optimal solutions correspond to the admissible states $S$. The state expansion of the Day--Sarkar graphons separates the admissible states from the lower-order states, while the synchronization lemma chooses a threshold cut at which the optimal non-admissible states are negligible. The surviving admissible contribution is then realized by the three-step threshold graphon $T_\beta(q)$.

In the finite threshold proof of Theorem~\ref{thmweightedthreshold}, the same fractional independence program controls the residual graphs after a high-degree core has been selected. The discrete state bounds separate states according to $\rho_F(S)$, and the layer synchronization lemma eliminates the same-order non-admissible residual states. The remaining admissible contribution is encoded by common-neighbourhood data around the core and is realized by a chained threshold graph.

In the bipartite host proof of Theorem~\ref{thmbipmain}, the relevant linear program is the finite-kernel scale $\kappa_H(n,m)$, or equivalently its logarithmic form $\alpha_p(H)$. The Shearer projection bound gives the matching upper bound for this scale, while the lower bound is obtained by rounding a continuous bipartite blow-up kernel. The comparison with the quasi-complete bipartite graph is then reduced to Hall's matching condition on the smaller colour classes, with positive Hall deficiency producing the counterexamples.

\medskip

The rest of the paper is organized as follows. Section~\ref{secprelim} collects basic notation and auxiliary inequalities. Section~\ref{secstates} develops the finite-kernel linear programs and the state decomposition used throughout the paper. Section~\ref{secdsproof} proves Theorem~\ref{thmds}. Section~\ref{secbpproof} proves Theorems~\ref{thmbp} and~\ref{thmweightedthreshold}. Section~\ref{secbipproof} proves Theorem~\ref{thmbipmain}. We conclude with examples and further remarks in Section~\ref{secexamples}.

\section{Preliminaries}\label{secprelim}

In this section, we provide the common definitions and estimates used in the three applications. Some of the finite-state notation below follows the viewpoint from \cite{KuangSunWangZeng2026}.

\subsection{Basic notations}
Throughout this paper, if $A\subseteq[0,1]^r$ is measurable, then $\lambda(A)$ denotes its $r$-dimensional Lebesgue measure, with $r$ understood from context. For a graph $H$ and a graphon $W$, the homomorphism density is
\[
t(H,W)=\int_{[0,1]^{V(H)}}\prod_{uv\in E(H)}W(x_u,x_v)\prod_{u\in V(H)}\dd x_u.
\]
If $D\subseteq[0,1]$ is measurable, then $W_D(x,y)=\1_{\max\{x,y\}\in D}$ is a threshold graphon. Its edge density is $\int_D 2t\dd t$. Threshold graph limits are represented by such graphons up to the usual weak isomorphism of graphons \cite{DiaconisHolmesJanson2008}.

A finite graph $T$ is a threshold graph if its vertices can be ordered so that each vertex is either isolated from all previous vertices or adjacent to all previous vertices when it is added. Equivalently, $T$ is obtained from the empty graph by adding isolated and dominating vertices one at a time \cite{MahadevPeled1995}.


 Let $F$ be a graph without isolated vertices. The fractional matching number of $F$ is
\[
\nu^*(F)=\max\left\{\sum_{e\in E(F)}\lambda_e\ \middle|\ \lambda_e\ge0\ \text{for every }e\in E(F),\ \sum_{e\ni v}\lambda_e\le1\ \text{for every }v\in V(F)\right\}.
\]
We shall repeatedly use the following dual form of fractional independence.

\begin{lemma}\label{lemdualalpha}
If $F$ has no isolated vertices, then $\nu^*(F)=|V(F)|-\alpha^*(F)$.
\end{lemma}

\begin{proof}
The fractional vertex cover program is
\[
\min\left\{\sum_{v\in V(F)}z_v\ \middle|\ z_u+z_v\ge1\ \text{for every }uv\in E(F),\ z_v\ge0\ \text{for every }v\in V(F)\right\}.
\]
It is dual to the fractional matching program, so its optimum equals $\nu^*(F)$ by linear programming duality \cite{Schrijver1986}. If $z$ is feasible, then replacing each $z_v$ by $\min\{z_v,1\}$ preserves feasibility and does not increase the objective. Hence one may assume $0\le z_v\le1$ for all $v$. The change of variables $x_v=1-z_v$ converts the vertex cover objective into $|V(F)|-\sum_vx_v$ and converts the constraints into $x_u+x_v\le1$ with $x_v\ge0$. The result follows.
\end{proof}

We call a vector half-integral if all its coordinates belong to
$\{0,\frac12,1\}$.

\begin{lemma}\label{lemhalfintegral}
For every graph $F$, every vertex of the polytope
\[
P_F=\left\{x\in\R^{V(F)}\ \middle|\
0\le x_v\le1\ \text{for every }v\in V(F),\
x_u+x_v\le1\ \text{for every }uv\in E(F)
\right\}
\]
is half-integral.
\end{lemma}

\begin{proof}
Let $x$ be an extreme point. Let $A=\{v\mid 0<x_v<1\}$, and let $J$ be the graph on $A$ whose edges are the edges $uv\in E(F)$ satisfying $x_u+x_v=1$. We show that every component of $J$ contains an odd cycle. Suppose that some component $C$ of $J$ is bipartite with bipartition $C_1\cup C_2$. Choose $\delta>0$ smaller than all positive coordinates on $C$, smaller than all numbers $1-x_v$ with $v\in C$, and smaller than all positive slacks of constraints incident with a vertex of $C$ that are not tight inside $J$. Increasing the coordinates on $C_1$ by $\delta$ and decreasing the coordinates on $C_2$ by $\delta$ gives a feasible point, and the opposite perturbation also gives a feasible point. Their midpoint is $x$, contradicting extremality. Hence every component of $J$ contains an odd cycle. On an odd cycle the equations $x_u+x_v=1$ force every coordinate to be $1/2$, and connectedness then forces every coordinate in the component to be $1/2$. Coordinates outside $A$ are integral. This proves the lemma.
\end{proof}

For $S\subseteq V(F)$, let $R_S=V(F)\setminus S$. In the induced graph $F[R_S]$, let $Z_S$ be the set of isolated vertices, and let $B_S=F[R_S\setminus Z_S]$. The graph $B_S$ is empty or has no isolated vertices. Define $\alpha^*(\emptyset)=0$ and
\[
\rho_F(S)=|Z_S|+\alpha^*(B_S).
\]
The set $S$ is admissible for $F$ if $\rho_F(S)=\alpha^*(F)$ and $B_S$ is empty or balanced, where balanced means $\alpha^*(B_S)=|V(B_S)|/2$.

\begin{lemma}\label{lemstatemonotonicity}
For every graph $F$ without isolated vertices and every $S\subseteq V(F)$, one has $\rho_F(S)\le\alpha^*(F)$.
\end{lemma}

\begin{proof}
Let $x$ be an optimal fractional independent vector on $B_S$. Define $\widetilde x_v=0$ for $v\in S$, define $\widetilde x_v=1$ for $v\in Z_S$, and define $\widetilde x_v=x_v$ for $v\in V(B_S)$. An edge inside $B_S$ satisfies the fractional independence constraint by the choice of $x$. An edge incident with $S$ has one endpoint of weight $0$. No edge of $F[R_S]$ is incident with a vertex of $Z_S$. Thus $\widetilde x$ is feasible for $F$, and its objective value is $\rho_F(S)$. The inequality follows by maximality of $\alpha^*(F)$.
\end{proof}

\begin{lemma}\label{lemadmissiblematching}
If $S$ is admissible for $F$, then there is a matching from $S$ into $Z_S$ such that each $s\in S$ is matched to a neighbour of $s$ in $F$.
\end{lemma}

\begin{proof}
Define $x_v=0$ for $v\in S$, define $x_v=1$ for $v\in Z_S$, and define $x_v=1/2$ for $v\in V(B_S)$. Since $S$ is admissible, this is a feasible fractional independent vector on $F$ with total weight $\alpha^*(F)$. Suppose that Hall's condition fails for the bipartite graph between $S$ and $Z_S$ whose edges are inherited from $F$. Then there is $X\subseteq S$ with $|N_F(X)\cap Z_S|<|X|$. Change the coordinates in $X$ from $0$ to $1/2$ and the coordinates in $N_F(X)\cap Z_S$ from $1$ to $1/2$, leaving all other coordinates unchanged. The new vector is feasible, since edges inside $X$ have weight sum at most $1$, edges from $X$ to $S\setminus X$ have weight sum $1/2$, edges from $X$ to $B_S$ have weight sum $1$, and every edge from $X$ to $Z_S$ ends in $N_F(X)\cap Z_S$. The total weight increases by $(|X|-|N_F(X)\cap Z_S|)/2$, a contradiction. Hall's theorem gives the matching.
\end{proof}

\subsection{Projection and graphon inequalities}

The following Shearer-type projection inequality is standard; we include the proof for convenience, in the notation used later. See \cite{ChungGrahamFranklShearer1986} and also \cite{KuangSunWangZeng2026}.

\begin{lemma}\label{lemshearer}
Let $\Omega$ be a finite set of assignments on a finite coordinate set $J$. Let $J_1,\ldots,J_s\subseteq J$, and let $\gamma_1,\ldots,\gamma_s\ge0$ satisfy $\sum_{i: j\in J_i}\gamma_i\ge1$ for every $j\in J$. If $\pi_i(\Omega)$ is the projection of $\Omega$ to the coordinates in $J_i$, then
\[
|\Omega|\le\prod_{i=1}^s|\pi_i(\Omega)|^{\gamma_i}.
\]
\end{lemma}

\begin{proof}
The case $\Omega=\emptyset$ is trivial, so assume $\Omega\ne\emptyset$. Let $X$ be uniformly distributed on $\Omega$, and fix an ordering of $J$. For $j\in J$, write $X_{<j}$ for the coordinates preceding $j$, and write $J_{i,<j}$ for the elements of $J_i$ preceding $j$. By the chain rule and monotonicity of conditional entropy,
\[
\mathrm H(X_{J_i})=\sum_{j\in J_i}\mathrm H(X_j\mid X_{J_{i,<j}})\ge\sum_{j\in J_i}\mathrm H(X_j\mid X_{<j}).
\]
Multiplying by $\gamma_i$, summing over $i$, and by the assumption of $\gamma_i$, we have
\[
\sum_{i=1}^s\gamma_i\mathrm H(X_{J_i})\ge\sum_{j\in J}\left(\sum_{i: j\in J_i}\gamma_i\right)\mathrm H(X_j\mid X_{<j})\ge\sum_{j\in J}\mathrm H(X_j\mid X_{<j})=\mathrm H(X).
\]
Since $\mathrm H(X)=\log|\Omega|$ and $\mathrm H(X_{J_i})\le\log|\pi_i(\Omega)|$, the result follows. 
\end{proof}

We shall use the following standard weighted form of H\"older's inequality.

\begin{lemma}\label{lemholderweighted}
Let $(\Omega,\mu)$ be a probability space. Let $f_1,\ldots,f_s$ be nonnegative measurable functions on $\Omega$, and let $\lambda_1,\ldots,\lambda_s\ge 0$ satisfy $ \sum_{i=1}^s \lambda_i\le 1.$
Then
\[
\int_\Omega \prod_{i=1}^s f_i^{\lambda_i}\dd\mu
\le
\prod_{i=1}^s \left(\int_\Omega f_i\dd\mu\right)^{\lambda_i}.
\]
\end{lemma}

For $A\subseteq [n]$ and $x=(x_1,\cdots,x_n)$, we define $x_A=(x_j)_{j\in A}$.

\begin{lemma}\label{lemfinner}
Let $\Omega_1,\ldots,\Omega_n$ be probability spaces. For each $i$ in a finite index set $I$, let $\emptyset\neq A_i\subseteq[n]$, let $f_i:\prod_{j\in A_i}\Omega_j\rightarrow [0,+\infty)$ be a nonnegative measurable function of the coordinates in $A_i$, and let $\lambda_i\ge0$. If $\sum_{i: j\in A_i}\lambda_i\le1$ for every $j\in[n]$, then
\[
\int_{\Omega_1\times\cdots\times\Omega_n}\prod_{i\in I}f_i(x_{A_i})^{\lambda_i}\dd x\le\prod_{i\in I}\left(\int f_i\right)^{\lambda_i}.
\]
Here $\int f_i$ denotes integration over the coordinates in $A_i$.
\end{lemma}

\begin{proof}
We prove the lemma by induction on $n$. First suppose $n=1$. Since each $A_i$ is nonempty, we have $A_i=\{1\}$ for every $i\in I$. Then by the assumption we have $\sum_{i\in I}\lambda_i = \sum_{i:\,1\in A_i}\lambda_i \le 1.$
Applying Lemma~\ref{lemholderweighted} on the probability space $\Omega_1$, we obtain
\[
\int_{\Omega_1}\prod_{i\in I}f_i(x_1)^{\lambda_i}\dd x_1
\le
\prod_{i\in I}\left(\int_{\Omega_1}f_i(x_1)\dd x_1\right)^{\lambda_i}.
\]
Now we assume that $n\ge 2$ and that the result has been proved for products of $n-1$ probability spaces. Let $I_n=\{i\in I:\ n\in A_i\}$ and $I_0=I\setminus I_n$. Then $\sum_{i\in I_n}\lambda_i=\sum_{i:n\in A_i}\lambda_i\leq 1$. For each $i\in I_n$, define
\[
g_i(x_{A_i\setminus\{n\}})
=
\int_{\Omega_n} f_i(x_{A_i})\dd x_n.
\]
If $A_i=\{n\}$, then $g_i$ is a constant. By Fubini's theorem, we have $\int g_i=\int f_i$.  By Lemma~\ref{lemholderweighted}, applied on the probability space $\Omega_n$ with $x_1,\ldots,x_{n-1}$ fixed, we have
\[
\begin{aligned}
\int_{\Omega_n}
\prod_{i\in I}f_i(x_{A_i})^{\lambda_i}\dd x_n
&=
\left(\prod_{i\in I_0}f_i(x_{A_i})^{\lambda_i}\right)
\int_{\Omega_n}
\prod_{i\in I_n}f_i(x_{A_i})^{\lambda_i}\dd x_n \\
&\le
\left(\prod_{i\in I_0}f_i(x_{A_i})^{\lambda_i}\right)
\prod_{i\in I_n}
\left(\int_{\Omega_n}f_i(x_{A_i})\dd x_n\right)^{\lambda_i} \\
&=
\left(\prod_{i\in I_0}f_i(x_{A_i})^{\lambda_i}\right)
\left(\prod_{i\in I_n}g_i(x_{A_i\setminus\{n\}})^{\lambda_i}\right).
\end{aligned}
\]
Now we integrate both sides over the first $n-1$ coordinates. Notice that for any $j\in[n-1]$, we have $\sum_{\substack{i\in I_0\\ j\in A_i}}\lambda_i
+
\sum_{\substack{i\in I_n\\ j\in A_i\setminus\{n\}}}\lambda_i
\le
\sum_{i:\,j\in A_i}\lambda_i
\le 1.$
The terms corresponding to $A_i=\{n\}$ are constant after integration over $x_n$, so the induction hypothesis is applied to the remaining factors. By Fubini's theorem and the induction hypothesis, we have
\[
\begin{aligned}
\int_{\Omega_1\times\cdots\times\Omega_n}
\prod_{i\in I}f_i(x_{A_i})^{\lambda_i}\dd x &\le \int_{\Omega_1\times\cdots\times\Omega_{n-1}} \left(\prod_{i\in I_0}f_i(x_{A_i})^{\lambda_i}\right) \left(\prod_{i\in I_n}g_i(x_{A_i\setminus\{n\}})^{\lambda_i}\right) \dd x_1\cdots\dd x_{n-1} \\ &\le \prod_{i\in I_0}\left(\int f_i\right)^{\lambda_i} \prod_{i\in I_n}\left(\int g_i\right)^{\lambda_i} \\ &= \prod_{i\in I_0}\left(\int f_i\right)^{\lambda_i} \prod_{i\in I_n}\left(\int f_i\right)^{\lambda_i}  =\prod_{i\in I}\left(\int f_i\right)^{\lambda_i},
\end{aligned}
\]
which proves the lemma.
\end{proof}

\begin{lemma}\label{lemsparsegraphon}
Let $F$ be a graph without isolated vertices, let $A\subseteq[0,1]$ have measure $\tau>0$, and let $U$ be a graphon supported on $A\times A$. Let $p=\int_{A^2}U$. Then
\[
t(F,U)\le p^{|V(F)|-\alpha^*(F)}\tau^{2\alpha^*(F)-|V(F)|}.
\]
In particular, if $U$ is any graphon of edge density $p$, then $t(F,U)\le p^{|V(F)|-\alpha^*(F)}$.
\end{lemma}

\begin{proof}
After rescaling $A$ to a probability space, the restricted graphon has edge density $p/\tau^2$, and the homomorphism density in the original space is $\tau^{|V(F)|}$ times the density in the rescaled space. It remains to prove the case $\tau=1$. Let $(\lambda_e)_{e\in E(F)}$ be a fractional matching with $\sum_e\lambda_e=|V(F)|-\alpha^*(F)$. Since $0\le U\le1$ and $0\le\lambda_e\le1$, one has $\prod_{e\in E(F)}U(x_e)\le\prod_{e\in E(F)}U(x_e)^{\lambda_e}$. By Lemma \ref{lemfinner},
\[
t(F,U)\le\prod_{e\in E(F)}\left(\int_{[0,1]^2}U\right)^{\lambda_e}=p^{\sum_e\lambda_e}=p^{|V(F)|-\alpha^*(F)}.
\]
This proves the lemma.
\end{proof}

\subsection{Homomorphisms and embeddings}

We shall use the following finite decomposition of homomorphisms according to which vertices of $H$ are identified. For a partition $\cP$ of $V(H)$ into nonempty independent sets, define a graph $H/\cP$ as follows. The vertices of $H/\cP$ are the parts of $\cP$, and two distinct parts $P,Q\in\cP$ are adjacent if and only if there is an edge of $H$ with one endpoint in $P$ and the other endpoint in $Q$. Let $\cQ(H)$ be the finite set of graphs obtained in this way.

\begin{lemma}\label{lemquotient}
For every fixed graph $H$, there are nonnegative integers $a_J=a_J(H)$, indexed by $J\in\cQ(H)$, such that for every simple graph $G$ one has $\homg(H,G)=\sum_{J\in\cQ(H)}a_J\Emb(J,G).$ If $H$ has no isolated vertices and $a_J>0$, then $J$ has no isolated vertices.
\end{lemma}

\begin{proof}
Every homomorphism $\phi:H\to G$ determines a partition $\cP$ of $V(H)$ by putting two vertices $u,v\in V(H)$ in the same part exactly when $\phi(u)=\phi(v)$. Since $G$ has no loops, no edge of $H$ can have both endpoints in the same part, so every part of $\cP$ is an independent set. Let $J=H/\cP$. The map from $V(J)$ to $V(G)$ sending a part $P\in\cP$ to the common value of $\phi$ on $P$ is well-defined, injective, and edge preserving. Thus it gives an embedding of $J$ into $G$. Conversely, suppose $J\in\cQ(H)$. A way to obtain $J$ from $H$ by such a partition, together with an embedding of $J$ into $G$, determines a homomorphism from $H$ to $G$. The number of ways to obtain a fixed graph $J$ from $H$ depends only on $H$ and $J$; denote this number by $a_J$. Therefore
\[
\homg(H,G)=\sum_{J\in\cQ(H)}a_J\Emb(J,G).
\]
Finally, assume that $H$ has no isolated vertices and $a_J>0$. Then $J=H/\cP$ for some partition $\cP$ as above. For each part $P\in\cP$, choose a vertex $v\in P$. Since $H$ has no isolated vertices, $v$ has a neighbour $w$ in $H$. Since $P$ is an independent set, we must have $w\notin P$. Hence the vertex $P$ of $H/\cP$ is adjacent to the part containing $w$. Thus no vertex of $J$ is isolated.
\end{proof}


For a fixed bipartite graph $H$ without isolated vertices and integers $n,m$ with $n\le m\le n^2$, define
\[
\kappa_H(n,m)=\max\left\{\prod_{v\in V(H)}a_v\ \middle|\ 1\le a_v\le n\ \text{for every }v\in V(H),\ a_ua_v\le m\ \text{for every }uv\in E(H)\right\}.
\]
The variables are real variables. If $p\in[1,2]$, define
\[
\alpha_p(H)=\max\left\{\sum_{v\in V(H)}x_v\ \middle|\ 0\le x_v\le1\ \text{for every }v\in V(H),\ x_u+x_v\le p\ \text{for every }uv\in E(H)\right\}.
\]

\begin{lemma}\label{lemkappalp}
Assume $n\ge2$ and $n\le m\le n^2$. Let $p=\log m/\log n$. Then $\kappa_H(n,m)=n^{\alpha_p(H)}$. Moreover, $\log\kappa_H(n,m)$ equals the minimum of
\[
(\log n)\sum_{v\in V(H)}\mu_v+(\log m)\sum_{uv\in E(H)}\lambda_{uv}
\]
over all nonnegative $\mu_v$ and $\lambda_{uv}$ satisfying
\[
\mu_v+\sum_{uv\in E(H)}\lambda_{uv}\ge1\quad\text{for every }v\in V(H).
\]
\end{lemma}

\begin{proof}
For every feasible vector $(a_v)_{v\in V(H)}$ in the definition of $\kappa_H(n,m)$, define $x_v=\log a_v/\log n$. Then $0\le x_v\le 1$ for every $v$, and $x_u+x_v\le p$ for every edge $uv$. Conversely, every feasible vector $(x_v)$ for $\alpha_p(H)$ gives the feasible vector $a_v=n^{x_v}$ for $\kappa_H(n,m)$. This proves $\kappa_H(n,m)=n^{\alpha_p(H)}$. For the dual form, define $z_v=\log a_v$. The logarithmic primal program is
\[
\max\left\{\sum_{v\in V(H)}z_v\ \middle|\ 0\le z_v\le\log n\ \text{for every }v,\ z_u+z_v\le\log m\ \text{for every }uv\in E(H)\right\}.
\]
The minimization problem is the linear programming dual. Strong duality applies because the primal is feasible and bounded \cite{Schrijver1986}.
\end{proof}


We finish the section with some natural scales used in Section~\ref{secbpproof}. In the discrete sparse threshold theorem, for fixed $m,n$ we define $\Lambda=\min\{n,m\}$,$ \delta=\frac m\Lambda$, and $Q=\sqrt m.$
Thus $m=\delta\Lambda$ and $Q=\sqrt{\delta\Lambda}$. The following lemma simply holds. 

\begin{lemma}\label{lemscaleseparation}
If $m\to\infty$, $m\le\binom n2$, and $m=o(n^{3/2})$, then $\delta/\Lambda\to0$, $\delta=o(Q)$, and $Q=o(\Lambda)$. Consequently, if $r=O(\delta)$, then $r=o(Q)$, $r^2=o(m)$, and $rQ=o(m)$.
\end{lemma}

\begin{proof}
If $m\le n$, then $\Lambda=m$ and $\delta=1$, so the assertions follow from $m\to\infty$. Suppose that $m>n$. Then $\Lambda=n$ and $\delta=m/n$. The hypothesis $m=o(n^{3/2})$ gives $\delta/\Lambda=m/n^2=o(n^{-1/2})$. Since $\delta/Q=Q/\Lambda=\sqrt{\delta/\Lambda}$, we get $\delta=o(Q)$ and $Q=o(\Lambda)$. The three consequences follow from $r=O(\delta)$ and $m=\delta\Lambda$.
\end{proof}

\section{Synchronization}\label{secstates}

This section develops the finite state language used in the graphon and finite graph parts of the proof. The first subsection identifies the leading states with the optimal half integral fractional independent vectors. The second subsection gives the exact state expansion for the Day--Sarkar threshold graphons. The third subsection records the finite graph analogues used later in the thresholdization argument. 

\subsection{Leading states and half integral optima}\label{subsecstatedictionary}

For a graph $F$ without isolated vertices and a set $S\subseteq V(F)$, recall that $R_S=V(F)\setminus S$, that $Z_S$ is the set of isolated vertices in $F[R_S]$, and that $B_S=F[R_S\setminus Z_S]$. We associate to $S$ the vector $\psi_S:V(F)\to\{0,1/2,1\}$ given by
\[
\psi_S(v)=
\begin{cases}
0,&v\in S,\\
1,&v\in Z_S,\\
1/2,&v\in V(B_S).
\end{cases}
\]
The next lemma shows that admissible states are exactly the optimal half integral vectors that occur in the polynomial $P_F(q)$.

\begin{lemma}\label{lemstatedictionary}
Let $F$ be a graph without isolated vertices. A set $S\subseteq V(F)$ is admissible if and only if $\psi_S\in\Phi^*(F)$. Moreover, the maps $S\mapsto\psi_S$ and $\phi\mapsto\phi^{-1}(0)$ are inverse bijections between the admissible sets $S\subseteq V(F)$ and the vectors in $\Phi^*(F)$.
\end{lemma}

\begin{proof}
Assume first that $S$ is admissible. If $uv\in E(F)$ and both $u,v$ lie in $R_S$, then neither endpoint lies in $Z_S$ unless the other endpoint lies in $S$, which is impossible because both endpoints lie in $R_S$. Hence every edge of $F[R_S]$ with both endpoints in $R_S$ is an edge of $B_S$, and both endpoints have $\psi_S$ value $1/2$. If one endpoint lies in $S$, then its $\psi_S$ value is $0$. It follows that $\psi_S(u)+\psi_S(v)\le1$ for every $uv\in E(F)$. Since $S$ is admissible, $B_S$ is empty or satisfies $\alpha^*(B_S)=|V(B_S)|/2$, and $\rho_F(S)=\alpha^*(F)$. Hence
\[
\sum_{v\in V(F)}\psi_S(v)=|Z_S|+\frac{|V(B_S)|}{2}=|Z_S|+\alpha^*(B_S)=\rho_F(S)=\alpha^*(F).
\]
Thus $\psi_S\in\Phi^*(F)$.

Conversely, let $\phi\in\Phi^*(F)$ and define $S=\phi^{-1}(0)$. If $v\in\phi^{-1}(1)$ and $vw\in E(F)$, then $\phi(w)=0$, since otherwise $\phi(v)+\phi(w)>1$. Therefore every vertex in $\phi^{-1}(1)$ is isolated in $F[V(F)\setminus S]$, so $\phi^{-1}(1)\subseteq Z_S$. If $v\in Z_S$ and $\phi(v)=1/2$, then every neighbour of $v$ lies in $S$, and changing $\phi(v)$ from $1/2$ to $1$ preserves all edge constraints and increases the total weight. This contradicts the optimality condition $\sum_v\phi(v)=\alpha^*(F)$. Hence $Z_S=\phi^{-1}(1)$. It follows that $V(B_S)=\phi^{-1}(1/2)$. The vector $\psi_S$ is therefore equal to $\phi$. Since $\psi_S$ is feasible and has total weight $\alpha^*(F)$, the argument in the first paragraph gives
\[
|Z_S|+\frac{|V(B_S)|}{2}=\alpha^*(F).
\]
The vector that assigns weight $1/2$ to every vertex of $B_S$ is feasible for $B_S$, so $\alpha^*(B_S)\ge |V(B_S)|/2$. By Lemma \ref{lemstatemonotonicity}, $|Z_S|+\alpha^*(B_S)=\rho_F(S)\le\alpha^*(F)$. Combining these two inequalities with the equality above, we obtain that $\alpha^*(B_S)=|V(B_S)|/2$ and $\rho_F(S)=\alpha^*(F)$. Thus $S$ is admissible. The last assertion follows because the proof has shown that every $\phi\in\Phi^*(F)$ is equal to $\psi_{\phi^{-1}(0)}$, while $\psi_S^{-1}(0)=S$ for every admissible $S$.
\end{proof}

The next lemma gives a positive gap between the leading finite states and all other states. This gap is used to justify that only admissible states contribute to the first order asymptotics.

\begin{lemma}\label{lemstategap}
Let $F$ be a fixed graph without isolated vertices, and set $h=|V(F)|$. There is a positive real number $\gamma_F>0$ such that, for every $S\subseteq V(F)$ and every $U\subseteq Z_S$, either $S$ is admissible and $U=\emptyset$, or
\[
|S|+\frac{|V(B_S)|+|U|}{2}\ge h-\alpha^*(F)+\gamma_F.
\]
\end{lemma}

\begin{proof}
For fixed $S$ and $U\subseteq Z_S$, define
\[
\eta(S,U)=|S|+\frac{|V(B_S)|+|U|}{2}.
\]
Since $h=|S|+|Z_S|+|V(B_S)|$ and $\rho_F(S)=|Z_S|+\alpha^*(B_S)$, we have
\[
\eta(S,U)-(h-\alpha^*(F))=\alpha^*(F)-\rho_F(S)+\alpha^*(B_S)-\frac{|V(B_S)|}{2}+\frac{|U|}{2}.
\]
Lemma \ref{lemstatemonotonicity} gives $\alpha^*(F)-\rho_F(S)\ge0$. The vector assigning $1/2$ to every vertex of $B_S$ is feasible for $B_S$, so $\alpha^*(B_S)-|V(B_S)|/2\ge0$. Hence $\eta(S,U)\ge h-\alpha^*(F)$ for all $S$ and $U$. Equality holds if and only if $\rho_F(S)=\alpha^*(F)$, $\alpha^*(B_S)=|V(B_S)|/2$, and $U=\emptyset$. This is exactly the condition that $S$ is admissible and $U=\emptyset$. There are only finitely many pairs $(S,U)$, so the minimum of $\eta(S,U)-(h-\alpha^*(F))$ over all pairs for which equality does not hold is positive. This minimum is the required $\gamma_F$.
\end{proof}

\subsection{The threshold graphon state expansion}\label{subsecthresholdstateexpansion}

We next translate the finite state language into an exact expansion for the Day--Sarkar graphons from~\cite{DaySarkar2021}. For $0<\beta<1$ and $0\le q\le1$, define
\[
\ell_{\beta,q}=\sqrt{\beta}q,\qquad r_{\beta,q}=1-\sqrt{1-\beta(1-q^2)},\qquad z_{\beta,q}=1-\ell_{\beta,q}-r_{\beta,q}.
\]
The three intervals are $L_{\beta,q}=[0,\ell_{\beta,q}]$, $R_{\beta,q}=[1-r_{\beta,q},1]$, and $Z_{\beta,q}=[0,1]\setminus(L_{\beta,q}\cup R_{\beta,q})$, up to endpoints of measure zero. The set $L_{\beta,q}$ is the lower clique interval, $R_{\beta,q}$ is the dominating interval, and $Z_{\beta,q}$ is the middle independent interval.

\begin{lemma}\label{lemexactthresholdexpansion}
Let $F$ be a fixed graph without isolated vertices. For every $0<\beta<1$ and every $0\le q\le1$,
\[
t(F,T_\beta(q))=\sum_{S\subseteq V(F)}\sum_{U\subseteq Z_S}r_{\beta,q}^{|S|}\ell_{\beta,q}^{|V(B_S)|+|U|}z_{\beta,q}^{|Z_S|-|U|}.
\]
\end{lemma}

\begin{proof}
    Partition $[0,1]$ into $L=L_{\beta,q}$, $Z=Z_{\beta,q}$, and $R=R_{\beta,q}$, ignoring endpoints of measure zero. For any assignment $\tau:V(F)\to\{L,Z,R\}$, we define \[A_{\tau}:=\{x\in[0,1]^{V(F)}:x_v\in \tau(v), \forall v\in V(F)\}.\] An assignment $\tau$ is called \textit{allowed} if for every edge $uv\in E(F)$, we have $\{\tau(u),\tau(v)\}\cap\{R\}\ne\emptyset$ or $\tau(u)=\tau(v)=L$. Then we have \[\begin{aligned}
        t(F,T_\beta(q))&= \int_{[0,1]^{V(F)}}\prod_{uv\in E(F)}T_{\beta}(q)(x_u,x_v)\prod_{u\in V(F)}dx_u \\ &=\sum_{\tau:V(F)\to \{L,Z,R\}}\int_{A_{\tau}}\prod_{uv\in E(F)}T_{\beta}(q)(x_u,x_v)\prod_{u\in V(F)}dx_u\\ &=\sum_{\tau:V(F)\to \{L,Z,R\} \text{ allowed}}\lambda(A_\tau).
    \end{aligned}\] 
    The last equality holds since if $\tau$ is allowed, then the integrand is identically $1$ on $A_\tau$, while if $\tau$ is not allowed, then some edge of $F$ is assigned to a forbidden pair of intervals and the integrand is identically $0$ on $A_\tau$.

    It remains to describe the allowed assignments. Let $\tau$ be allowed and set $S=\tau^{-1}(R)$. In the residual graph $F[R_S]$, no vertex assigned to $Z$ can be incident to an edge, since neither a $Z$--$Z$ edge nor a $Z$--$L$ edge is allowed. Hence $\tau^{-1}(Z)\subseteq Z_S$. Therefore every non-isolated vertex of $F[R_S]$ must be assigned to $L$, so $V(B_S)\subseteq \tau^{-1}(L)$. Thus $\tau$ is determined by $S$ and by the subset $
U=\tau^{-1}(L)\cap Z_S\subseteq Z_S. $
Explicitly, the assignment has the form
\[
\tau^{-1}(R)=S,\qquad \tau^{-1}(L)=V(B_S)\cup U,\qquad \tau^{-1}(Z)=Z_S\setminus U.
\]
Conversely, for every $S\subseteq V(F)$ and every $U\subseteq Z_S$, the assignment defined by the above three identities is allowed. Since every edge incident with $S$ has an endpoint assigned to $R$, and every edge in $F[R_S]$ has both endpoints in $V(B_S)$ and hence both endpoints assigned to $L$. Therefore the allowed assignments are in bijection with pairs $(S,U)$ such that $S\subseteq V(F)$ and $U\subseteq Z_S$. For the assignment corresponding to such a pair, the measure of $A_\tau$ is $
\lambda(A_\tau)
=
r_{\beta,q}^{|S|}\ell_{\beta,q}^{|V(B_S)|+|U|}z_{\beta,q}^{|Z_S|-|U|}. $ Therefore, we have
\[
t(F,T_\beta(q))
=
\sum_{S\subseteq V(F)}\sum_{U\subseteq Z_S}
r_{\beta,q}^{|S|}\ell_{\beta,q}^{|V(B_S)|+|U|}z_{\beta,q}^{|Z_S|-|U|},
\]
as required.
\end{proof}

The preceding exact expansion immediately gives the first order threshold asymptotics. The leading coefficient is precisely the polynomial $P_F(q)$ defined in Section \ref{secintro}, with $F$ in place of $H$.

\begin{lemma}\label{lemleadingthresholdexpansion}
Let $F$ be a fixed graph without isolated vertices, and let $\gamma_F$ be as in Lemma \ref{lemstategap}. As $\beta\rightarrow0$,
\[
t(F,T_\beta(q))=\beta^{|V(F)|-\alpha^*(F)}\left(P_F(q)+O_F(\beta^{\min\{\gamma_F,1/2\}})\right)
\]
uniformly for $0\le q\le1$.
\end{lemma}

\begin{proof}
Write $h=|V(F)|$, $\alpha=\alpha^*(F)$, and $d=h-\alpha$. For $0\le q\le1$ and $0<\beta<1/2$, we have $ 0\le \ell_{\beta,q}\le \beta^{1/2}$, $
0\le r_{\beta,q}\le \beta$, and $
0\le z_{\beta,q}\le 1.$
Moreover,
\[ r_{\beta,q} = \frac{\beta(1-q^2)}{1+\sqrt{1-\beta(1-q^2)}} = \beta\left(\frac{1-q^2}{2}+O(\beta)\right) \] uniformly in $q$, while $\ell_{\beta,q}=\beta^{1/2}q$, and $z_{\beta,q}=1+O(\beta^{1/2})$ uniformly in $q$.
By Lemma~\ref{lemexactthresholdexpansion},
\[
\begin{aligned}
t(F,T_\beta(q))
&=
\sum_{S\subseteq V(F)}\sum_{U\subseteq Z_S}
r_{\beta,q}^{|S|}
\ell_{\beta,q}^{|V(B_S)|+|U|}
z_{\beta,q}^{|Z_S|-|U|} \\
&=
\sum_{\substack{S\subseteq V(F)\\ S\text{ admissible}}}
r_{\beta,q}^{|S|}
\ell_{\beta,q}^{|V(B_S)|}
z_{\beta,q}^{|Z_S|}
+
\sum_{\substack{S\subseteq V(F),\,U\subseteq Z_S\\ S\text{ not admissible or }U\ne\emptyset}}
r_{\beta,q}^{|S|}
\ell_{\beta,q}^{|V(B_S)|+|U|}
z_{\beta,q}^{|Z_S|-|U|}.
\end{aligned}
\]
We first estimate the second sum. Since $r_{\beta,q}\le\beta$, $\ell_{\beta,q}\le\beta^{1/2}$, and $z_{\beta,q}\le1$, each term in the second sum is at most $ \beta^{|S|+(|V(B_S)|+|U|)/2}.$
By Lemma~\ref{lemstategap}, whenever $S$ is not admissible or $U\ne\emptyset$, we have
$|S|+\frac{|V(B_S)|+|U|}{2}\ge d+\gamma_F. $ Since there are only finitely many pairs $(S,U)$, it follows that
\[
\sum_{\substack{S\subseteq V(F),\,U\subseteq Z_S\\ S\text{ not admissible or }U\ne\emptyset}}
r_{\beta,q}^{|S|}
\ell_{\beta,q}^{|V(B_S)|+|U|}
z_{\beta,q}^{|Z_S|-|U|}
=
O_F(\beta^{d+\gamma_F})
\]
uniformly in $q$. It remains to evaluate the first sum. Fix an admissible set $S$, and define $s_S=|S|$, $ b_S=|V(B_S)|$, and $c_S=|Z_S|.$
Since $S$ is admissible, we have $\alpha=|Z_S|+\alpha^*(B_S)=c_S+\frac{b_S}{2}. $ Thus $d=h-\alpha=s_S+\frac{b_S}{2}.$
Using the estimates above, we obtain that
\[
\begin{aligned}
r_{\beta,q}^{s_S}
\ell_{\beta,q}^{b_S}
z_{\beta,q}^{c_S}
&=
\left[\beta\left(\frac{1-q^2}{2}+O(\beta)\right)\right]^{s_S}
(\beta^{1/2}q)^{b_S}
(1+O(\beta^{1/2}))^{c_S} \\
&=
\beta^{s_S+b_S/2}
\left(
\left(\frac{1-q^2}{2}\right)^{s_S}q^{b_S}
+
O_F(\beta^{1/2})
\right) \\
&=
\beta^d
\left(
\left(\frac{1-q^2}{2}\right)^{s_S}q^{b_S}
+
O_F(\beta^{1/2})
\right),
\end{aligned}
\]
uniformly in $q$. Summing over all admissible $S$ gives
\[
\begin{aligned}
\sum_{\substack{S\subseteq V(F)\\ S\text{ admissible}}}
r_{\beta,q}^{|S|}
\ell_{\beta,q}^{|V(B_S)|}
z_{\beta,q}^{|Z_S|}
&=
\beta^d
\left(
\sum_{\substack{S\subseteq V(F)\\ S\text{ admissible}}}
\left(\frac{1-q^2}{2}\right)^{|S|}
q^{|V(B_S)|}
+
O_F(\beta^{1/2})
\right).
\end{aligned}
\]
By Lemma~\ref{lemstatedictionary}, the map $S\mapsto\psi_S$ is a bijection from admissible sets to $\Phi^*(F)$, with $|\psi_S^{-1}(0)|=|S|$ and $|\psi_S^{-1}(1/2)|=|V(B_S)|$.
Therefore,
\[ 
\sum_{\substack{S\subseteq V(F)\\ S\text{ admissible}}} \left(\frac{1-q^2}{2}\right)^{|S|} q^{|V(B_S)|} = P_F(q).
\]
Combining the estimates for the two sums, we obtain
\[
t(F,T_\beta(q))
=
\beta^d\left(P_F(q)+O_F(\beta^{1/2})\right)
+
O_F(\beta^{d+\gamma_F})
=
\beta^d\left(P_F(q)+O_F(\beta^{\min\{\gamma_F,1/2\}})\right),
\]
uniformly for $0\le q\le1$.
\end{proof}

\subsection{Finite threshold kernels}\label{subsecfinitethresholdkernels}

The finite graph proof uses the same state language with three discrete scales. The next graph is the finite threshold kernel corresponding to the three intervals in the graphon calculation.

\begin{definition}\label{defthreeblockkernel}
Let $R,L,Z$ be pairwise disjoint finite sets. Define $K(R,L,Z)$ to be the graph on $R\cup L\cup Z$ whose edges are all pairs with at least one endpoint in $R$, together with all pairs contained in $L$.
\end{definition}

The following elementary observation follows immediately by listing first the vertices of $L$, then those of $Z$, and finally those of $R$; we omit the proof.

\begin{lemma}\label{lemthreeblockthreshold}
For all pairwise disjoint finite sets $R,L,Z$, the graph $K(R,L,Z)$ is a threshold graph.
\end{lemma}

The next lemma is the finite realization of a state. It is used later to build threshold graphs that realize the same finite kernel states as the graphon construction.

\begin{lemma}\label{lemfinitestaterealization}
Let $F$ be a fixed graph and let $S\subseteq V(F)$. Let $K(R,L,Z)$ be a three block threshold kernel. The number of labelled embeddings of $F$ into $K(R,L,Z)$ that map $S$ into $R$, $V(B_S)$ into $L$, and $Z_S$ into $Z$ is at least
\[
(|R|)_{|S|}(|L|)_{|V(B_S)|}(|Z|)_{|Z_S|},
\]
where $(x)_k=x(x-1)\cdots(x-k+1)$ for $k\ge1$ and $(x)_0=1$. In particular, if $|R|,|L|,|Z|\ge2|V(F)|$, then this number is at least $2^{-|V(F)|}|R|^{|S|}|L|^{|V(B_S)|}|Z|^{|Z_S|}$.
\end{lemma}

\begin{proof}
Choose injectively the images of the vertices in $S$ inside $R$, the images of the vertices in $V(B_S)$ inside $L$, and the images of the vertices in $Z_S$ inside $Z$. This gives exactly $(|R|)_{|S|}(|L|)_{|V(B_S)|}(|Z|)_{|Z_S|}$ choices. The chosen images are all distinct because $R,L,Z$ are disjoint and the choices inside each block are injective. It remains to check that every edge of $F$ is mapped to an edge of $K(R,L,Z)$. If an edge has an endpoint in $S$, then its image has an endpoint in $R$, so it is an edge of $K(R,L,Z)$. If an edge has no endpoint in $S$, then it is an edge of $F[V(F)\setminus S]$. No vertex of $Z_S$ is incident with such an edge by the definition of $Z_S$, so both endpoints lie in $V(B_S)$ and their images lie in $L$. Since $L$ is a clique in $K(R,L,Z)$, this edge is preserved. Thus every choice gives a labelled embedding. If $|R|,|L|,|Z|\ge2|V(F)|$, then $(x)_k\ge x^k/2^k$ for every $0\le k\le |V(F)|$ and every $x\ge2|V(F)|$, because each factor $x-i$ with $0\le i<k$ is at least $x/2$. Multiplying the three estimates gives the final lower bound.
\end{proof}

We next provide a finite upper bound that converts residual embeddings into the same fractional exponents as in the graphon estimate. This will be applied after separating a small promoted set from the remaining graph.

\begin{lemma}\label{lemfinitesparseembedding}
Let $F$ be a fixed nonempty graph without isolated vertices. If $G$ is a graph with $N\ge1$ vertices and $m$ edges, then
\[
\homg(F,G)\le (2m)^{|V(F)|-\alpha^*(F)}N^{2\alpha^*(F)-|V(F)|}.
\]
The same upper bound holds for $\Emb(F,G)$.
\end{lemma}

\begin{proof}
The embedding bound follows from the homomorphism bound, so it is enough to prove the homomorphism bound. If $m=0$, then $\homg(F,G)=0$, because $F$ has at least one edge. Assume $m>0$. Partition $[0,1]$ into intervals $I_v$ of length $1/N$, one interval for each $v\in V(G)$. Let $W_G$ be the graphon that is equal to $1$ on $I_u\times I_v$ and on $I_v\times I_u$ whenever $uv\in E(G)$, and equal to $0$ elsewhere. Then $\int_{[0,1]^2}W_G=2m/N^2$. For every homomorphism $\varphi:F\to G$, the box $\prod_{x\in V(F)}I_{\varphi(x)}$ contributes exactly $N^{-|V(F)|}$ to $t(F,W_G)$, and no other box contributes unless the corresponding map is a homomorphism. Hence $t(F,W_G)=\homg(F,G)/N^{|V(F)|}$. Lemma \ref{lemsparsegraphon} gives $t(F,W_G)\le(2m/N^2)^{|V(F)|-\alpha^*(F)}$. Multiplying by $N^{|V(F)|}$ gives the displayed estimate.
\end{proof}

The next estimate is the discrete state decomposition around a promoted vertex set. The set $A$ will later be chosen as a high degree or synchronized core.

\begin{lemma}\label{lemdiscretestateupper}
Let $F$ be a fixed graph without isolated vertices, let $G$ be a graph, and let $A\subseteq V(G)$. Let $N=|V(G)|$ and let $m_A=e(G[V(G)\setminus A])$. For $S\subseteq V(F)$, let $\Emb_S(F,G,A)$ be the set of labelled embeddings $\varphi:F\to G$ such that $\varphi^{-1}(A)=S$. Then
\[
|\Emb_S(F,G,A)|\le |A|^{|S|}N^{|Z_S|}(2m_A)^{|V(B_S)|-\alpha^*(B_S)}N^{2\alpha^*(B_S)-|V(B_S)|},
\]
with the convention that the last two factors are equal to $1$ when $B_S$ is empty.
\end{lemma}

\begin{proof}
Let $\varphi\in\Emb_S(F,G,A)$. The images of the vertices in $S$ have at most $|A|^{|S|}$ possible choices. The images of the vertices in $Z_S$ have at most $N^{|Z_S|}$ possible choices. Since $\varphi^{-1}(A)=S$, every vertex in $V(B_S)$ is mapped into $V(G)\setminus A$, and the restriction of $\varphi$ to $V(B_S)$ is a homomorphism from $B_S$ to $G[V(G)\setminus A]$. We ignore injectivity and the edges from $S$ to $V(F)\setminus S$, since ignoring constraints can only increase the count. If $B_S$ is empty, this gives the displayed estimate. If $B_S$ is nonempty, then $B_S$ has no isolated vertices by definition, and Lemma \ref{lemfinitesparseembedding} applied to $B_S$ and $G[V(G)\setminus A]$ gives
\[
\homg(B_S,G[V(G)\setminus A])\le(2m_A)^{|V(B_S)|-\alpha^*(B_S)}N^{2\alpha^*(B_S)-|V(B_S)|}.
\]
Multiplying the three independent upper bounds gives the result.
\end{proof}

The finite thresholdization proof repeatedly separates vertices that are too heavy for the current residual graph. The following elementary lemma is the precise form of this light or promote step.

\begin{lemma}\label{lemlightorpromote}
Let $G$ be a graph with $m$ edges, let $A\subseteq V(G)$, and let $\theta>0$. Let $U=V(G)\setminus A$, and let $P=\{v\in U: d_{G[U]}(v)>\theta\}$. Then $|P|<2m/\theta$, every vertex of $U\setminus P$ has at most $\theta$ neighbours in $U$, and $e(G[U\setminus P])<e(G[U])-\frac{\theta|P|}{2}$ whenever $P\ne\emptyset$.
\end{lemma}

\begin{proof}
By the definition of $P$,
\[
\theta|P|<\sum_{v\in P}d_{G[U]}(v)\le2e(G[U])\le2m.
\]
This proves $|P|<2m/\theta$. Every vertex of $U\setminus P$ has at most $\theta$ neighbours in $U$ by the definition of $P$. It remains to prove the edge decrease. Let $e_P$ be the number of edges of $G[U]$ with both endpoints in $P$, and let $e_{P,U\setminus P}$ be the number of edges of $G[U]$ with one endpoint in $P$ and the other endpoint in $U\setminus P$. Then $\sum_{v\in P}d_{G[U]}(v)=2e_P+e_{P,U\setminus P}$, while the number of edges removed when passing from $G[U]$ to $G[U\setminus P]$ is $e_P+e_{P,U\setminus P}$. Hence
\[
e(G[U])-e(G[U\setminus P])=e_P+e_{P,U\setminus P}\ge\frac12(2e_P+e_{P,U\setminus P})=\frac12\sum_{v\in P}d_{G[U]}(v)>\frac{\theta|P|}{2}.
\]
This proves the inequality.
\end{proof}

The last lemma allows later sections to discard all states that are already known to be lower order.

\begin{lemma}\label{lemfinitebookkeeping}
Let $\cI$ be a finite set, let $X_i(n)\ge0$ for $i\in\cI$, and let $Y(n)>0$. If $X_i(n)=o(Y(n))$ for every $i$ in a subset $\cI_0\subseteq\cI$, then
\[
\sum_{i\in\cI_0}X_i(n)=o(Y(n)).
\]
If $\sum_{i\in\cI}X_i(n)\ge cY(n)$ for some constant $c>0$ and all sufficiently large $n$, then there is an index $i\in\cI$ and an infinite sequence of $n$ such that $X_i(n)\ge cY(n)/|\cI|$ along that sequence.
\end{lemma}

\begin{proof}
For the first assertion, write $|\cI_0|=s$. If $s=0$, the assertion is trivial. If $s>0$, then for every $\eps>0$ and every $i\in\cI_0$ there is $n_i$ such that $X_i(n)\le\eps Y(n)/s$ for all $n\ge n_i$. For $n\ge\max_i n_i$, summing over $i\in\cI_0$ gives $\sum_{i\in\cI_0}X_i(n)\le\eps Y(n)$. This proves the first assertion. For the second assertion, suppose the conclusion is false. Then for every $i\in\cI$ there is $n_i$ such that $X_i(n)<cY(n)/|\cI|$ for all $n\ge n_i$. For $n\ge\max_i n_i$, summing over $i\in\cI$ gives $\sum_{i\in\cI}X_i(n)<cY(n)$, contradicting the hypothesis. Hence the conclusion holds.
\end{proof}

\section{Proof of Theorem \ref{thmds}}\label{secdsproof}

Throughout this section let $h=|V(H)|$, let $\alpha=\alpha^*(H)$, and let $d=h-\alpha$. If $H$ is empty, then we are done. We therefore assume that $H$ is nonempty. Since $H$ has no isolated vertices, every coordinate in every feasible fractional independent vector is at most $1$, and at least one edge of $H$ forces the sum of two coordinates to be at most $1$. Hence $\alpha<h$ and $d>0$.

We first state the precise finite graph theorem of Blekherman and Patel that is used in this section. The graphon reduction needed below will then be derived as a corollary.

\begin{proposition}[Blekherman and Patel, reformulated from~\cite{BlekhermanPatel2024}]\label{propbpfinite}
Let $F$ be a fixed graph and let $c\in[0,1]$. Let $\cC_c$ be the class of finite graphs $G$ satisfying $t(K_2,G)\le c$, and let $\cT$ be the class of finite threshold graphs. Then
\[
\limsup_{G\in\cC_c}t(F,G)=\limsup_{G\in\cC_c\cap\cT}t(F,G),
\]
where $\limsup_{G\in\cA}f(G)$ means $\lim_{N\to\infty}\sup\{f(G)\mid G\in\cA,\ |V(G)|\ge N\}$.
\end{proposition}

\begin{proof}
This is exactly Theorem 3 of Blekherman and Patel \cite[Theorem 3]{BlekhermanPatel2024}, written with the notation used in this paper.
\end{proof}

We also use the threshold graph limit representation of Diaconis, Holmes, and Janson~\cite{DiaconisHolmesJanson2008}.

\begin{lemma}[Diaconis, Holmes, and Janson~\cite{DiaconisHolmesJanson2008}]\label{lemthresholdlimitrepresentation}
Let $U$ be a graphon which is a graph limit of threshold graphs. Then there exists a measurable set $D\subseteq[0,1]$ such that $t(F,U)=t(F,W_D)$ for every finite graph $F$, where $ W_D(x,y)=\1_{\{\max\{x,y\}\in D\}}.$
\end{lemma}

The next corollary is the graphon form needed for the proof. It follows from Proposition \ref{propbpfinite}, the compactness of graphons, and the description of threshold graph limits.

\begin{corollary}\label{corgraphonthresholdreduction}
For every fixed graph $F$ and every $0\le\beta\le1$,
\[\begin{aligned}
    &\sup\{t(F,W)\mid W\text{ is a graphon and }t(K_2,W)\le\beta\}\\&=\sup\{t(F,W_D)\mid D\subseteq[0,1]\text{ is measurable and }2\int_Dt\dd t\le\beta\}.
\end{aligned}
\]
\end{corollary}

\begin{proof}
Let the left hand side be $L_\beta$ and the right hand side be $R_\beta$. Since every graphon $W_D$ satisfying $2\int_Dt\dd t\le\beta$ has edge density at most $\beta$, we have $R_\beta\le L_\beta$. We prove the reverse inequality. If $\beta=0$, then every graphon $W$ with $t(K_2,W)=0$ is zero almost everywhere. Hence $t(F,W)=0$ whenever $F$ has at least one edge, and the same value is attained by $W_\emptyset$. If $F$ has no edges, then $t(F,W)=1$ for every graphon $W$, and the equality is again trivial. Thus we may assume that $\beta>0$.

Let $W$ be a graphon with $t(K_2,W)\le\beta$, and fix $0<\eta<1$. Define $W_\eta=(1-\eta)W$. Then $t(K_2,W_\eta)\le(1-\eta)\beta<\beta$, and $t(F,W_\eta)=(1-\eta)^{|E(F)|}t(F,W)$. By the standard finite graph approximation theorem for graphons (see~\cite{Lovasz2012}), there is a sequence of finite graphs $G_N$ with $|V(G_N)|\to\infty$ such that $t(J,G_N)\to t(J,W_\eta)$ for every fixed graph $J$. In particular, $t(K_2,G_N)\to t(K_2,W_\eta)<\beta$ and $t(F,G_N)\to t(F,W_\eta)$. Hence $G_N\in\cC_\beta$ for all sufficiently large $N$, and therefore $\limsup_{G\in\cC_\beta}t(F,G)\ge t(F,W_\eta)$.

By Proposition \ref{propbpfinite}, $\limsup_{G\in\cC_\beta\cap\cT}t(F,G)\ge t(F,W_\eta)$. Thus for every integer $s\ge1$ there is a threshold graph $T_s$ with $|V(T_s)|\ge s$, $t(K_2,T_s)\le\beta$, and $t(F,T_s)\ge t(F,W_\eta)-1/s$. By graphon compactness, a subsequence of the graphons associated with $T_s$ converges to a graphon $U$. Homomorphism densities are continuous under graphon convergence, so $t(K_2,U)\le\beta$ and $t(F,U)\ge t(F,W_\eta)$. Since $U$ is a graph limit of threshold graphs, by Lemma~\ref{lemthresholdlimitrepresentation}, it is weakly isomorphic to a graphon $W_D$ with $D\subseteq[0,1]$ measurable. Since $t(K_2,W_D)=t(K_2,U)\le\beta$ and $t(K_2,W_D)=2\int_Dt\dd t$, this graphon is allowed in the definition of $R_\beta$. Hence $R_\beta\ge t(F,W_\eta)$.

Letting $\eta\rightarrow0$ gives $R_\beta\ge t(F,W)$. Taking the supremum over all graphons $W$ with $t(K_2,W)\le\beta$ gives $R_\beta\ge L_\beta$. This completes the proof.
\end{proof}

The three-step threshold graphons already give the desired lower bound and the exact threshold-family asymptotic.

\begin{lemma}\label{lemthreestepasymp}
As $\beta\rightarrow0$ one has
$\cM^T_H(\beta)=\beta^d(C_T(H)+o(1)).$
\end{lemma}

\begin{proof}
Lemma \ref{lemleadingthresholdexpansion}, applied with $F=H$, gives a number $\eta_H>0$ such that
\[
t(H,T_\beta(q))=\beta^d\left(P_H(q)+O_H(\beta^{\eta_H})\right)
\]
uniformly for $0\le q\le1$. Taking the supremum over $q\in[0,1]$ gives
\[
\cM^T_H(\beta)=\beta^d\left(\max_{0\le q\le1}P_H(q)+O_H(\beta^{\eta_H})\right)=\beta^d(C_T(H)+o(1)).
\]
\end{proof}

We now consider an arbitrary threshold graphon. Let $D\subseteq[0,1]$ be measurable and suppose that $2\int_Dt\dd t=\beta$. For $\tau\in(0,1]$, define $C_\tau=D\cap(\tau,1]$, define $D_\tau=D\cap[0,\tau]$, define $R_\tau=\int_{C_\tau}t\dd t$, and define $Y_\tau=\int_{D_\tau}t\dd t$. Then $R_\tau+Y_\tau=\beta/2$. Define $q_\tau\in[0,1]$ by $q_\tau^2=2Y_\tau/\beta$. Lebesgue measure is denoted by $\lambda$. For a point $x=(x_v)_{v\in V(H)}\in[0,1]^{V(H)}$, define $S_\tau(x)=\{v\in V(H)\mid x_v\in C_\tau\}$. We say that $x$ is a \textit{homomorphism point} for $W_D$ if $\prod_{uv\in E(H)}W_D(x_u,x_v)=1$.

The next claim says that, after the vertices lying in the upper cut $C_\tau$ are fixed, the non-isolated residual vertices must lie below the cut.

\begin{claim}\label{claimresidualprefix}
Let $S=S_\tau(x)$. If $x$ is a homomorphism point for $W_D$, then $x_u\le\tau$ for every $u\in V(B_S)$.
\end{claim}

\begin{proof}
Let $u\in V(B_S)$. Since $B_S$ has no isolated vertices, there is $v\in V(B_S)$ such that $uv\in E(H)$. Suppose by contradiction that $\max(x_u,x_v)>\tau$. Since $x$ is a homomorphism point, $W_D(x_u,x_v)=1$. By the definition of $W_D$, we have $\max(x_u,x_v)\in D$. Hence the endpoint attaining the maximum lies in $D\cap(\tau,1]=C_\tau$. This endpoint belongs to $S_\tau(x)=S$, contradicting $u,v\in V(B_S)\subseteq V(H)\setminus S$. Therefore $\max(x_u,x_v)\le\tau$, and in particular $x_u\le\tau$.
\end{proof}

The next claim shows that every state with $\rho_H(S)$ below the optimal value gives a smaller power of $\beta$. This is the first place where the finite state decomposition removes lower order terms.

\begin{claim}\label{claimlowstate}
For every fixed $\tau>0$, the measure of all homomorphism points $x$ for $W_D$ satisfying $\rho_H(S_\tau(x))<\alpha$ is $o_{H,\tau}(\beta^d)$ as $\beta\rightarrow0$, where $d=|V(H)|-\alpha^{*}(H)$.
\end{claim}

\begin{proof}
Define
\[
\cA=\{x\in[0,1]^{V(H)}: \prod_{uv\in E(H)}W_D(x_u,x_v)=1, \text{ and } \rho_H(S_\tau(x))<\alpha\}.
\]
It suffices to prove $\lambda(\cA)=o_{H,\tau}(\beta^d)$. For $S\subseteq V(H)$, define
\[
\cA_S=\{x\in[0,1]^{V(H)}: \prod_{uv\in E(H)}W_D(x_u,x_v)=1 \text{ and }S_\tau(x)=S\}.
\]
Then $\lambda(\cA)\le\sum_{\substack{S\subseteq V(H)\\ \rho_H(S)<\alpha}}\lambda(\cA_S).$
Fix such an $S$. Write $s=|S|$, $b=|V(B_S)|$, and $\alpha_B=\alpha^*(B_S)$. Every coordinate indexed by $S$ lies in $C_\tau$. Since $C_\tau\subseteq[\tau,1]$, we have $\tau\lambda(C_\tau)\le\int_{C_\tau}t\dd t\le\beta/2$, and hence $\lambda(C_\tau)\le\beta/(2\tau)$. By Claim~\ref{claimresidualprefix}, every coordinate indexed by $V(B_S)$ lies in $[0,\tau]$. Dropping all constraints involving $Z_S$, all constraints inside $S$, and all constraints between $S$ and $V(H)\setminus S$,
\[
\lambda(\cA_S)\le\lambda(C_\tau)^s\int_{[0,\tau]^{V(B_S)}}\prod_{uv\in E(B_S)}W_D(y_u,y_v)\prod_{u\in V(B_S)}\dd y_u.
\]
The last integral is interpreted as $1$ when $B_S$ is empty.

Assume first that $B_S$ is nonempty. Define $\widetilde W_\tau(x,y)=W_D(\tau x,\tau y)$ for $x,y\in[0,1]$. Then
\[
t(K_2,\widetilde W_\tau)=\frac{1}{\tau^2}\int_{[0,\tau]^2}W_D=\frac{2Y_\tau}{\tau^2}\le\frac{\beta}{\tau^2}.
\]
Therefore by Lemma~\ref{lemsparsegraphon}
\[
\int_{[0,\tau]^{V(B_S)}}\prod_{uv\in E(B_S)}W_D(y_u,y_v)\prod_{u\in V(B_S)}\dd y_u
=\tau^b t(B_S,\widetilde W_\tau)\le\tau^b\left(\frac{\beta}{\tau^2}\right)^{b-\alpha_B}
\]
The same bound is $1$ when $B_S$ is empty, under the stated convention. Hence
\[
\lambda(\cA_S)\le\left(\frac{\beta}{2\tau}\right)^s\beta^{b-\alpha_B}\tau^{2\alpha_B-b}=2^{-s}\tau^{-s+2\alpha_B-b}\beta^{s+b-\alpha_B}.
\]
Since $h=s+|Z_S|+b$ and $\rho_H(S)=|Z_S|+\alpha_B$, we have $s+b-\alpha_B=h-\rho_H(S)$. Thus $\lambda(\cA_S)=O_{H,\tau}(\beta^{h-\rho_H(S)})$. Because $\rho_H(S)<\alpha$, we have $h-\rho_H(S)>h-\alpha=d$, and consequently $\lambda(\cA_S)=o_{H,\tau}(\beta^d)$. There are only finitely many subsets $S\subseteq V(H)$, thus summing over all $S$ with $\rho_H(S)<\alpha$ gives $\lambda(\cA)=o_{H,\tau}(\beta^d)$.
\end{proof}

The admissible states have the same exponent as the desired answer. The next claim bounds their leading contribution by the polynomial $P_H$ evaluated at the cut parameter $q_\tau$.

\begin{claim}\label{claimadmcontribution}
For every fixed $\tau>0$, the measure of all homomorphism points $x$ for $W_D$ such that $S_\tau(x)$ is admissible is at most $\beta^dP_H(q_\tau)$.
\end{claim}

\begin{proof}
Define
\[
\cA=\left\{x\in[0,1]^{V(H)}\mid \prod_{uv\in E(H)}W_D(x_u,x_v)=1\text{ and }S_\tau(x)\text{ is admissible}\right\}.
\]
It suffices to prove $\lambda(\cA)\le\beta^dP_H(q_\tau)$. For each admissible $S\subseteq V(H)$, let
\[
\cA_S=\left\{x\in[0,1]^{V(H)}\mid \prod_{uv\in E(H)}W_D(x_u,x_v)=1\text{ and }S_\tau(x)=S\right\}.
\]
Then $\lambda(\cA)=\sum_{\substack{S\subseteq V(H)\\ S\text{ admissible}}}\lambda(\cA_S).$
Fix an admissible $S$. Let $B=B_S$, $s=|S|$, and $b=|V(B)|$. By Lemma~\ref{lemadmissiblematching}, there is a matching $m:S\to Z_S$ such that $sm(s)\in E(H)$ for every $s\in S$. For $x\in\cA_S$ and $s\in S$, one has $x_s\in C_\tau$ and $x_{m(s)}\notin C_\tau$. Since $sm(s)\in E(H)$, one also has $W_D(x_s,x_{m(s)})=1$. If $x_{m(s)}>x_s$, then
\[
x_{m(s)}=\max(x_s,x_{m(s)})\in D\cap(\tau,1]=C_\tau,
\]
which contradicts $x_{m(s)}\notin C_\tau$. Hence $x_{m(s)}\le x_s$ for every $s\in S$.

Dropping all constraints except the matched constraints and the constraints inside $B$,
\[
\lambda(\cA_S)\le\left(\prod_{s\in S}\int_{C_\tau}\int_0^{x_s}\dd x_{m(s)}\dd x_s\right)I_B,
\]
where $I_B=1$ if $B$ is empty, and otherwise
\[
I_B=\int_{[0,\tau]^{V(B)}}\prod_{uv\in E(B)}W_D(y_u,y_v)\prod_{u\in V(B)}\dd y_u.
\]
Since
\[
\int_{C_\tau}\int_0^{x_s}\dd x_{m(s)}\dd x_s=\int_{C_\tau}t\dd t=R_\tau,
\]
we have $\lambda(\cA_S)\le R_\tau^s I_B.$
If $B$ is nonempty, then by Claim~\ref{claimresidualprefix} all coordinates indexed by $V(B)$ lie in $[0,\tau]$. Since $S$ is admissible, $B$ is balanced, so $\alpha^*(B)=b/2$. Define $\widetilde W_\tau(x,y)=W_D(\tau x,\tau y)$ for $x,y\in[0,1]$. Then
\[
t(K_2,\widetilde W_\tau)=\frac{1}{\tau^2}\int_{[0,\tau]^2}W_D=\frac{2Y_\tau}{\tau^2}.
\]
By Lemma~\ref{lemsparsegraphon},
\[
I_B=\tau^b t(B,\widetilde W_\tau)\le\tau^b\left(\frac{2Y_\tau}{\tau^2}\right)^{b/2}=(2Y_\tau)^{b/2}.
\]
The same bound gives $I_B=1$ when $B$ is empty. Therefore $\lambda(\cA_S)\le R_\tau^s(2Y_\tau)^{b/2}.$
Using $R_\tau=\beta(1-q_\tau^2)/2$, $2Y_\tau=\beta q_\tau^2$, and $d=s+b/2$ from Lemma~\ref{lemstatedictionary}, we obtain
\[
\lambda(\cA_S)\le\beta^{s+b/2}\left(\frac{1-q_\tau^2}{2}\right)^s q_\tau^b\le\beta^d\left(\frac{1-q_\tau^2}{2}\right)^{|S|}q_\tau^{|V(B_S)|}.
\]
Summing over all admissible $S$ and using Lemma~\ref{lemstatedictionary},
\[
\lambda(\cA)\le\beta^d\sum_{\substack{S\subseteq V(H)\\ S\text{ admissible}}}\left(\frac{1-q_\tau^2}{2}\right)^{|S|}q_\tau^{|V(B_S)|}=\beta^dP_H(q_\tau).
\]
This proves the claim.
\end{proof}

It remains to choose one cut at which all optimal but non-admissible states are negligible. This is the synchronization step for threshold graphons.

\begin{lemma}\label{lemsynchronizationds}
For every $\eps>0$ there are an integer $R=R(H,\eps)$ and numbers $1=\tau_0>\tau_1>\cdots>\tau_R>0$, depending only on $H$ and $\eps$, with the following property. For every measurable $D\subseteq[0,1]$ satisfying $2\int_Dt\dd t=\beta$, if $\beta$ is sufficiently small, then there is $j\in\{0,\ldots,R-1\}$ such that the measure of all homomorphism points $x$ satisfying $\rho_H(S_{\tau_j}(x))=\alpha$ and such that $B_{S_{\tau_j}(x)}$ is non-balanced is at most $\eps\beta^d$, where $d=|V(H)|-\alpha^*(H)$.
\end{lemma}

\begin{proof}
If there is no set $S\subseteq V(H)$ with $\rho_H(S)=\alpha$ and non-balanced $B_S$, then every set counted in the statement is empty, and the conclusion holds with $R=1$, $\tau_0=1$, and any $\tau_1\in(0,1)$. Assume that such states exist. Define
\[
\eta_0=\min\{2\alpha^*(B_S)-|V(B_S)|\mid \rho_H(S)=\alpha\ \text{and}\ B_S\ \text{is non-balanced}\}.
\]
The number $\eta_0$ is positive, because $B_S$ is non-balanced exactly when $\alpha^*(B_S)>|V(B_S)|/2$, and there are only finitely many states. Choose an integer $R\ge\max\{1,4h/\eps\}$. Starting with $\tau_0=1$, choose $\tau_{j+1}\in(0,\tau_j)$ recursively so that
\[
2^h\tau_j^{-h}\tau_{j+1}^{\eta_0}\le\frac{\eps}{4R}
\]
for every $0\le j<R$. This is possible because $\eta_0>0$. For $0\le j<R$, let $\cN_j$ be the set of homomorphism points $x$ satisfying $\rho_H(S_{\tau_j}(x))=\alpha$ and such that $B_{S_{\tau_j}(x)}$ is non-balanced. Let $\cP_j$ be the subset of $\cN_j$ consisting of points with $S_{\tau_{j+1}}(x)\supsetneq S_{\tau_j}(x)$, and let $\mathcal{L}_j=\cN_j\setminus\cP_j$.

We first bound $\lambda(\mathcal{L}_j)$. Fix a state $S$ with $\rho_H(S)=\alpha$ and non-balanced $B=B_S$. Let $s=|S|$, let $b=|V(B)|$, and let $\alpha_B=\alpha^*(B)$. Let $\mathcal{L}_{j,S}$ be the subset of $\mathcal{L}_j$ with $S_{\tau_j}(x)=S$. Since $x\in\mathcal{L}_j$ gives $S_{\tau_{j+1}}(x)=S_{\tau_j}(x)$, the proof of Claim \ref{claimresidualprefix} with $\tau_{j+1}$ in place of $\tau$ shows that $x_u\le\tau_{j+1}$ for every $u\in V(B)$. The vertices in $S$ lie in $C_{\tau_j}$, whose measure is at most $\beta/(2\tau_j)$. The restriction of $W_D$ to $[0,\tau_{j+1}]^2$ has edge density mass at most $\beta$. Lemma \ref{lemsparsegraphon} applied on $[0,\tau_{j+1}]$ gives a factor at most $\beta^{b-\alpha_B}\tau_{j+1}^{2\alpha_B-b}$ for the variables of $B$. Dropping all restrictions on the variables in $Z_S$ and all cross restrictions gives
\[
\lambda(\mathcal{L}_{j,S})\le\left(\frac{\beta}{2\tau_j}\right)^s\beta^{b-\alpha_B}\tau_{j+1}^{2\alpha_B-b}.
\]
Since $\rho_H(S)=\alpha$, one has $d=h-\alpha=s+b-\alpha_B$. Since $B$ is non-balanced, $2\alpha_B-b\ge\eta_0$. Therefore
\[
\lambda(\mathcal{L}_{j,S})\le \beta^d\tau_j^{-h}\tau_{j+1}^{\eta_0}.
\]
There are at most $2^h$ choices for $S$. By the recursive choice of the cuts,
\[
\lambda(\mathcal{L}_j)\le 2^h\beta^d\tau_j^{-h}\tau_{j+1}^{\eta_0}\le \frac{\eps}{4R}\beta^d.
\]

We now bound the promoted parts. For each fixed homomorphism point $x$, the sets \[ S_{\tau_0}(x),S_{\tau_1}(x),\ldots,S_{\tau_R}(x)\] form an increasing sequence of subsets of $V(H)$, because $C_{\tau_j}\subseteq C_{\tau_{j+1}}$ when $\tau_{j+1}<\tau_j$. A strict inclusion can occur for at most $h$ indices. Hence
\[
\sum_{j=0}^{R-1}\lambda(\cP_j)\le h\,t(H,W_D).
\]
Lemma \ref{lemsparsegraphon}, applied to $H$ on the full interval $[0,1]$, gives $t(H,W_D)\le\beta^d$. Thus $\sum_{j=0}^{R-1}\lambda(\cP_j)\le h\beta^d$. Therefore there is $j\in\{0,\ldots,R-1\}$ such that $\lambda(\cP_j)\le h\beta^d/R\le\eps\beta^d/4$. For this index $j$,
\[
\lambda(\cN_j)\le\lambda(\mathcal{L}_j)+\lambda(\cP_j)\le\frac{\eps}{4R}\beta^d+\frac{\eps}{4}\beta^d\le\eps\beta^d.
\]
This proves the lemma.
\end{proof}

The previous lemmas give the upper bound for every threshold graphon of sufficiently small edge density.

\begin{lemma}\label{lemthresholdupperds}
For every $\eps>0$ there exists a $\beta_0=\beta_0(H,\eps)>0$ such that, whenever $0<\beta\le\beta_0$ and $D\subseteq[0,1]$ is measurable with $2\int_Dt\dd t=\beta$, one has
\[
t(H,W_D)\le\beta^d(C_T(H)+2\eps).
\]
\end{lemma}

\begin{proof}
Apply Lemma \ref{lemsynchronizationds} with $\eps$. Let $1=\tau_0>\tau_1>\cdots>\tau_R>0$ be the resulting cuts. Choose the index $j$ supplied by Lemma \ref{lemsynchronizationds}. The homomorphism points for $W_D$ split into three classes. The first class consists of points with $\rho_H(S_{\tau_j}(x))<\alpha$. Since $\tau_j$ is one of finitely many positive cuts depending only on $H$ and $\eps$, Claim \ref{claimlowstate} shows that the measure of this class is at most $\eps\beta^d$ for all sufficiently small $\beta$. The second class consists of points with $\rho_H(S_{\tau_j}(x))=\alpha$ and non-balanced residual $B_{S_{\tau_j}(x)}$. Lemma \ref{lemsynchronizationds} bounds the measure of this class by $\eps\beta^d$. The third class consists of points for which $S_{\tau_j}(x)$ is admissible. Claim \ref{claimadmcontribution} bounds the measure of this class by $\beta^dP_H(q_{\tau_j})\le\beta^dC_T(H)$. By Lemma \ref{lemstatemonotonicity}, no state satisfies $\rho_H(S)>\alpha$. Hence these three classes cover all homomorphism points. Summing the three bounds proves the lemma.
\end{proof}

We now finish the proof of the theorem.

\begin{proof}[Proof of Theorem \ref{thmds}]
By Lemma \ref{lemthreestepasymp}, $\cM^T_H(\beta)=\beta^d(C_T(H)+o(1))$. Since each $T_\beta(q)$ is a graphon of edge density $\beta$, by the definition of $\cM_H(\beta)$, we have that $\cM_H(\beta)\ge\cM^T_H(\beta)$.

It remains to prove the upper bound for $\cM_H(\beta)$. Fix $\eps>0$ and choose $\beta_0$ as in Lemma \ref{lemthresholdupperds}. Let $0<\beta\le\beta_0$ and let $W$ be a graphon with $t(K_2,W)\le\beta$. By Corollary \ref{corgraphonthresholdreduction}, for every $\eta>0$ there is a measurable set $D\subseteq[0,1]$ such that $2\int_Dt\dd t\le\beta$ and $t(H,W)\le t(H,W_D)+\eta\beta^d$. If $2\int_Dt\dd t=0$, then $W_D=0$ almost everywhere on every edge of $H$, and hence $t(H,W_D)=0$. Otherwise set $\theta=2\int_Dt\dd t$. Lemma \ref{lemthresholdupperds}, applied with $\theta$ in place of $\beta$, gives
\[
t(H,W_D)\le\theta^d(C_T(H)+2\eps)\le\beta^d(C_T(H)+2\eps).
\]
Thus $t(H,W)\le\beta^d(C_T(H)+2\eps)+\eta\beta^d$. Taking the supremum over all such $W$ and then letting $\eta\rightarrow0$ gives $\cM_H(\beta)\le\beta^d(C_T(H)+2\eps)$ for all sufficiently small $\beta$. Since $\eps>0$ is arbitrary, $\cM_H(\beta)\le\beta^d(C_T(H)+o(1))$. Together with the lower bound from the first paragraph and the threshold-family asymptotic from Lemma \ref{lemthreestepasymp}, this proves Theorem \ref{thmds}.
\end{proof}

\section{Proof of Theorems~\ref{thmbp} and~\ref{thmweightedthreshold}}\label{secbpproof}

Throughout this section, $\cF$ is a fixed finite family of graphs with no isolated vertices, $a_F\ge0$ for each $F\in\cF$, and at least one weight is positive. Write $v_F=|V(F)|$, $\alpha_F=\alpha^*(F)$, and $h_{\cF}=\max_{F\in\cF}|V(F)|$. We use the scales $\Lambda=\min\{n,m\}$, $\delta=m/\Lambda$, and $Q=\sqrt m$ from Lemma~\ref{lemscaleseparation}. Define
\[
A_F=\delta^{v_F-\alpha_F}\Lambda^{\alpha_F},\qquad A_{\cF}=\sum_{F\in\cF}a_FA_F.
\]
Graphs with zero weight will be ignored when constants depending on $\cF$ are chosen.

We first prove the stronger weighted embedding statement, Theorem~\ref{thmweightedthreshold}. Theorem~\ref{thmbp} will then follow from the finite quotient decomposition of homomorphisms into embeddings. The proof of Theorem~\ref{thmweightedthreshold} has four parts. We first establish the two-scale counting estimates at the scale $A_F$, together with matching threshold lower bounds. We then choose a synchronized high-degree core by a layer argument. With respect to this core, embeddings are decomposed by their state $S\subseteq V(F)$. States with $\rho_F(S)<\alpha_F$ are lower order, while states with $\rho_F(S)=\alpha_F$ and non-balanced residuals are made negligible by synchronization. The remaining admissible states are encoded by common-neighbourhood data around the core and are realized inside a chained threshold graph. This proves that the weighted embedding maximum over all graphs is at most $(1+o(1))$ times the corresponding maximum over threshold graphs. The opposite inequality is immediate because threshold graphs are among all graphs, and the passage from embeddings to homomorphisms gives Theorem~\ref{thmbp}.

We first prove a Shearer bound for balanced residual graphs. This is the discrete counterpart of the balanced graphon estimate.  We include the short proof for the reader's convenience.

\begin{lemma}\label{lembalancedembedding}
Let $B$ be a fixed graph that is empty or balanced. If $X$ is a graph with at most $m$ edges, then
\[
\Emb(B,X)\le(2m)^{|V(B)|/2}.
\]
\end{lemma}

\begin{proof}
If $B$ is empty, then $\Emb(B,X)=1$ and we are done. Assume that $B$ is nonempty. Since $B$ has no isolated vertices and is balanced, $\alpha^*(B)=|V(B)|/2$. By linear programming duality, there are nonnegative numbers $\lambda_e$, indexed by $e\in E(B)$, such that $\sum_{e\ni v}\lambda_e\ge1$ for every $v\in V(B)$ and $\sum_{e\in E(B)}\lambda_e=\alpha^*(B)=|V(B)|/2$. Let $\Omega=\Emb(B,X)$. For each edge $e=uv\in E(B)$, project an embedding to the ordered pair formed by the images of $u$ and $v$. The projection has size at most $2m$, because $X$ has at most $m$ unordered edges. By Lemma \ref{lemshearer} we have
\[
|\Omega|\le\prod_{e\in E(B)}(2m)^{\lambda_e}=(2m)^{|V(B)|/2}.
\]
This proves the lemma.
\end{proof}

The next lemma gives the half integral structure needed for non-balanced residuals. It also produces a map from the weight-one vertices to the weight-zero vertices, which is the source of the light-factor gain in the synchronization step.

\begin{lemma}\label{lemnonbalancedstructure}
Let $B$ be a fixed graph with no isolated vertices and $\alpha^*(B)>|V(B)|/2$. There is an optimal fractional independent vector $x$ with $x_v\in\{0,1/2,1\}$ for every $v\in V(B)$. If $D=\{v\mid x_v=0\}$, $I=\{v\mid x_v=1\}$, and $R=\{v\mid x_v=1/2\}$, then $I$ is independent, there are no edges between $I$ and $R$, $|I|>|D|$, there is a map $f:I\to D$ such that $if(i)\in E(B)$ for every $i\in I$, every $d\in D$ has a nonempty preimage under $f$, some $d_0\in D$ has at least two preimages, and $B[R]$ is empty or balanced.
\end{lemma}

\begin{proof}
The feasible polytope for $\alpha^*(B)$ is nonempty and bounded, since $B$ has no isolated vertices and every coordinate is at most $1$. Choose an optimal extreme point $x$. By Lemma \ref{lemhalfintegral}, every coordinate of $x$ lies in $\{0,1/2,1\}$. Define $D,I,R$ as in the statement. If two vertices of $I$ were adjacent, then the corresponding edge constraint would have left side $2$. If a vertex of $I$ were adjacent to a vertex of $R$, then the corresponding edge constraint would have left side $3/2$. Hence $I$ is independent and there are no edges between $I$ and $R$. Since $\sum_vx_v=|I|+|R|/2$ and $|V(B)|/2=(|D|+|I|+|R|)/2$, the strict inequality $\alpha^*(B)>|V(B)|/2$ gives $|I|>|D|$.

We prove Hall's condition for the bipartite graph between $D$ and $I$ whose edges are inherited from $B$. Suppose that there is $X\subseteq D$ with $|N_I(X)|<|X|$. Change the coordinates in $X$ from $0$ to $1/2$ and the coordinates in $N_I(X)$ from $1$ to $1/2$, leaving all other coordinates unchanged. We check feasibility. Edges inside $X$ have new weight sum $1$. Edges from $X$ to $D\setminus X$ have new weight sum $1/2$. Edges from $X$ to $R$ have new weight sum $1$. Edges from $X$ to $I\setminus N_I(X)$ do not exist by the definition of $N_I(X)$. Edges from $N_I(X)$ to $I$ do not exist because $I$ is independent. Edges from $N_I(X)$ to $R$ do not exist because there are no edges between $I$ and $R$. Edges from $N_I(X)$ to $D\setminus X$ have new weight sum $1/2$. All other edge constraints are unchanged. Thus the new vector is feasible. Its objective value increases by $(|X|-|N_I(X)|)/2$, contradicting the optimality of $x$. Therefore $|N_I(X)|\ge|X|$ for every $X\subseteq D$. By Hall's theorem, we obtain a matching from $D$ into $I$ that saturates $D$. Assign each matched vertex of $I$ to the vertex of $D$ to which it is matched. If $i\in I$ is unmatched, then $i$ has a neighbour in $D$, because $B$ has no isolated vertices, $I$ is independent, and there are no edges between $I$ and $R$. Assign such an unmatched $i$ to one of its neighbours in $D$. This gives a map $f:I\to D$ with $if(i)\in E(B)$ for every $i\in I$, and every vertex of $D$ receives at least one preimage. Since $|I|>|D|$, some vertex of $D$ receives at least two preimages.

It remains to prove that $B[R]$ is empty or balanced. If a vertex $r\in R$ were isolated in $B[R]$, then $r$ would have no neighbour in $I$ and no neighbour in $R$. Changing $x_r$ from $1/2$ to $1$ would preserve all edge constraints, since all neighbours of $r$ would lie in $D$, and would increase the objective value. This is impossible. Hence $B[R]$ is empty or has no isolated vertices. If $B[R]$ is not empty and is not balanced, then there is a feasible fractional independent vector $y$ on $B[R]$ with $\sum_{r\in R}y_r>|R|/2$. Keeping the values on $D\cup I$ and replacing the values on $R$ by $y$ gives a feasible fractional independent vector on $B$, because there are no edges between $I$ and $R$ and all edges from $D$ to $R$ have the endpoint in $D$ at weight $0$. This new vector has larger total weight than $x$, a contradiction. Therefore $B[R]$ is empty or balanced.
\end{proof}

Recall that $v_F=|V(F)|$, $\alpha_F=\alpha^*(F)$, $\Lambda=\min\{n,m\}$, and $\delta=m/\Lambda$. Recall that we also write $A_F=\delta^{v_F-\alpha_F}\Lambda^{\alpha_F}$ and $A_{\cF}=\sum_{F\in\cF}a_FA_F$. The following estimate is the two-scale counting bound used for residual graphs. The first part gives the correct $\delta$ and $\Lambda$ scale. The second part gives an additional factor when the vertices that should use the $\Lambda$ scale are forced to have degree less than $\rho\Lambda$.

\begin{lemma}\label{lemtwoscale}
Let $B$ be a fixed graph with no isolated vertices. Let $b=|V(B)|$ and $\alpha_B=\alpha^*(B)$. There is a constant $C_B$ such that every graph $X$ with $e(X)\le m$ and maximum degree at most $\Lambda$ satisfies
\[
\Emb(B,X)\le C_B\delta^{b-\alpha_B}\Lambda^{\alpha_B}.
\]
If $B$ is non-balanced and $D,I,R,f$ are chosen as in Lemma \ref{lemnonbalancedstructure}, then the number of embeddings of $B$ in $X$ whose vertices in $D$ are all mapped to vertices of degree less than $\rho\Lambda$ is at most $C_B\rho\delta^{b-\alpha_B}\Lambda^{\alpha_B}$ for every $0<\rho\le1$.
\end{lemma}

\begin{proof}
If $B$ is balanced, then $\alpha_B=b/2$. By Lemma~\ref{lembalancedembedding},
\[
\Emb(B,X)\le (2m)^{b/2}=2^{b/2}(\delta\Lambda)^{b/2}=2^{b/2}\delta^{b-\alpha_B}\Lambda^{\alpha_B}.
\]
This proves the first estimate in the balanced case. Now we assume that $B$ is non-balanced. Choose $D,I,R,f$ as in Lemma~\ref{lemnonbalancedstructure}. For $d_0\in D$, put $q_{d_0}=|f^{-1}(d_0)|$. Then $q_{d_0}\ge1$ and $\sum_{d_0\in D}q_{d_0}=|I|$. We count a larger class of maps by keeping only the edge constraints $if(i)$ for $i\in I$ and the edge constraints inside $B[R]$. For each $d_0\in D$, the number of choices for the image of $d_0$ and the images of the vertices in $f^{-1}(d_0)$ is at most
\[
\sum_{x\in V(X)}d_X(x)^{q_{d_0}}
\le
\Lambda^{q_{d_0}-1}\sum_{x\in V(X)}d_X(x)
\le
2m\Lambda^{q_{d_0}-1}
=
2\delta\Lambda^{q_{d_0}}.
\]
Since $B[R]$ is empty or balanced, by Lemma~\ref{lembalancedembedding}
\[
\Emb(B[R],X)\le C_B(2m)^{|R|/2}\le C_B(\delta\Lambda)^{|R|/2},
\]
with the value $1$ when $R=\emptyset$. Therefore
\[ \Emb(B,X) \le
C_B(\delta\Lambda)^{|R|/2}\prod_{d_0\in D}\left(2\delta\Lambda^{q_{d_0}}\right)\le C_B\delta^{|R|/2+|D|}\Lambda^{|R|/2+\sum_{d_0\in D}q_{d_0}}= C_B\delta^{|D|+|R|/2}\Lambda^{|I|+|R|/2}.
\]
By Lemma~\ref{lemnonbalancedstructure}, we have $\alpha_B=|I|+\frac{|R|}{2}$ and $ b-\alpha_B=|D|+\frac{|R|}{2}.$ Thus $\Emb(B,X)\le C_B\delta^{b-\alpha_B}\Lambda^{\alpha_B}.$

It remains to prove the second assertion. Let $\Emb_\rho(B,X)$ denote the embeddings of $B$ in $X$ in which every vertex of $D$ is mapped to a vertex of degree less than $\rho\Lambda$. By the same enlarged count, we obtain
\[
\Emb_\rho(B,X)
\le
C_B(\delta\Lambda)^{|R|/2}
\prod_{d_0\in D}
\sum_{\substack{x\in V(X), d_X(x)<\rho\Lambda}}
d_X(x)^{q_{d_0}}.
\]
For each $d_0\in D$,
\[ \sum_{\substack{x\in V(X), d_X(x)<\rho\Lambda}}d_X(x)^{q_{d_0}} \le (\rho\Lambda)^{q_{d_0}-1}\sum_{x\in V(X)}d_X(x)\le 2m\rho^{q_{d_0}-1}\Lambda^{q_{d_0}-1}= 2\rho^{q_{d_0}-1}\delta\Lambda^{q_{d_0}}.
\]
Hence
\[
\begin{aligned}
\Emb_\rho(B,X)
&\le
C_B(\delta\Lambda)^{|R|/2}
\prod_{d_0\in D}
\left(2\rho^{q_{d_0}-1}\delta\Lambda^{q_{d_0}}\right)\\
&\le
C_B\rho^{\sum_{d_0\in D}(q_{d_0}-1)}
\delta^{|D|+|R|/2}
\Lambda^{|I|+|R|/2}\\
&=
C_B\rho^{|I|-|D|}
\delta^{b-\alpha_B}
\Lambda^{\alpha_B}.
\end{aligned}
\]
Since $|I|>|D|$ and $0<\rho\le1$, we have $\rho^{|I|-|D|}\le\rho$. Therefore $\Emb_\rho(B,X)\le C_B\rho\delta^{b-\alpha_B}\Lambda^{\alpha_B}.$
\end{proof}

\begin{lemma}\label{lemglobaltwoscale}
There is a constant $C_{\cF}>0$ such that every graph $G$ with $|V(G)|\le n$ and $e(G)\le m$ satisfies $\Emb(F,G)\le C_{\cF}A_F$ for every $F\in\cF$ and
\[
\sum_{F\in\cF}a_F\Emb(F,G)\le C_{\cF}A_{\cF}.
\]
\end{lemma}

\begin{proof}
Since $|V(G)|\le n$ and $e(G)\le m$, the maximum degree of $G$ is at most $\min\{n,m\}=\Lambda$. Lemma \ref{lemtwoscale}, applied with $X=G$ and $B=F$, gives $\Emb(F,G)\le C_F\delta^{v_F-\alpha_F}\Lambda^{\alpha_F}=C_FA_F$ for every $F\in\cF$. Since $\cF$ is finite, increasing the constant gives $\Emb(F,G)\le C_{\cF}A_F$ for every $F\in\cF$. Multiplying by $a_F$ and summing over $F\in\cF$ gives the second desired estimate.
\end{proof}

We next show that threshold graphs already attain the correct scale $A_{\cF}$. This lower bound is used at the end to convert additive errors into relative errors.

\begin{lemma}\label{lemlowerscale}
There is a constant $c_{\cF}>0$ such that, for all sufficiently large $n$,
\[
\max_{\substack{|V(T)|\le n\\ e(T)\le m\\ T\in\cT}}\sum_{F\in\cF}a_F\Emb(F,T)\ge c_{\cF}A_{\cF}=c_{\cF}\sum_{F\in\cF}a_F\left(\frac{m}{\min\{n,m\}}\right)^{v_F-\alpha_F}\min\{n,m\}^{\alpha_F}.
\]
\end{lemma}

\begin{proof}
It is enough to prove that, for each fixed $F$ with $a_F>0$, there is a threshold graph $T_F$ satisfying $|V(T_F)|\le n$ and $e(T_F)\le m$ with $\Emb(F,T_F)\ge c_FA_F$. Since there are only finitely many positive-weight graphs, the largest term $a_FA_F$ is at least $A_{\cF}/s$, where $s$ is the number of positive-weight graphs. The desired constant is then obtained by taking the minimum of the finitely many constants $c_F/s$.

Fix $F$. If $F$ is balanced, let $T_F$ be a clique on $\floor{cQ}$ vertices, where $c>0$ is sufficiently small. Since $Q^2=m$, the edge number of this clique is at most $m$ for all sufficiently large $n$. Lemma \ref{lemscaleseparation} gives $Q=o(\Lambda)\le n$, so the vertex number is at most $n$ for all sufficiently large $n$. Every injection from $V(F)$ into this clique is an embedding, and hence $\Emb(F,T_F)\ge c_FQ^{v_F}=c_Fm^{v_F/2}=c_FA_F$, because $\alpha_F=v_F/2$.

Assume next that $F$ is non-balanced. Choose $D,I,R$ from Lemma \ref{lemnonbalancedstructure}. Let $C$ be a clique of size $\floor{c_1\delta}+c_0$, where $c_0>|D|$ is fixed. Let $P$ be an independent set of size $\floor{c_2\Lambda}$. Let $W$ be a clique of size $\floor{c_3Q}$. Construct a threshold graph by first creating the clique $W$, then adding all vertices of $P$ as isolated vertices, and then adding all vertices of $C$ as dominating vertices. In the resulting graph, $C$ is a clique, $W$ is a clique, all edges between $C$ and $P\cup W$ are present, and there are no edges between $P$ and $W$ or inside $P$. Choose $c_1,c_2,c_3>0$ sufficiently small, depending only on $F$. The edges between $C$ and $P$ use at most a fixed small multiple of $\delta\Lambda=m$, and the edges inside $W$ use at most a fixed small multiple of $Q^2=m$. The edges inside $C$ and between $C$ and $W$ are $O(\delta^2)+O(\delta Q)=o(m)$ by Lemma \ref{lemscaleseparation}. Thus $e(T_F)\le m$ for all sufficiently large $n$. The vertex number is at most $c_2\Lambda+o(\Lambda)\le n$ after decreasing $c_2$ if necessary.

Map the vertices in $D$ injectively into $C$, map the vertices in $I$ injectively into $P$, and map the vertices in $R$ injectively into $W$. All required edges of $F$ are present in $T_F$, because $I$ is independent, there are no edges between $I$ and $R$, the set $C$ is joined to $P\cup W$, and $W$ is a clique. The number of such embeddings is at least
\[
c_F\delta^{|D|}\Lambda^{|I|}Q^{|R|}=c_F\delta^{|D|+|R|/2}\Lambda^{|I|+|R|/2}=c_F\delta^{v_F-\alpha_F}\Lambda^{\alpha_F}=c_FA_F.
\]
This completes the proof.
\end{proof}

The next lemma is the finite synchronization step. It chooses one degree layer at which all same-order non-admissible residual states have negligible total contribution.

\begin{lemma}\label{lemlayersynchronization}
For every $\eta>0$, one can choose an integer $R$ and numbers $1=\tau_0>\tau_1>\cdots>\tau_R>0$, depending only on $\cF$ and $\eta$, such that the following holds. For every graph $G$ with $|V(G)|\le n$ and $e(G)\le m$, define $U_0=\emptyset$ and $U_j=\{x\in V(G)\mid d_G(x)\ge\tau_j\Lambda\}$ for $1\le j\le R$. Then there is $j\in\{0,\ldots,R-1\}$ such that the total weighted number of embeddings $\phi:F\hookrightarrow G$ satisfying $\rho_F(S_j(\phi))=\alpha_F$ and such that $S_j(\phi)$ is not admissible is at most $\eta A_{\cF}$, where $S_j(\phi)=\phi^{-1}(U_j)$.
\end{lemma}

\begin{proof}
For each pair $(F,S)$ such that $\rho_F(S)=\alpha_F$ and $B_S$ is non-balanced, fix once and for all a half integral partition $D_S,I_S,R_S$ of $B_S$ supplied by Lemma \ref{lemnonbalancedstructure}. Let $C$ be a constant larger than all constants in Lemma \ref{lemtwoscale}, all factors $2^{|S|}$, all positive weights $a_F$, the number of possible pairs $(F,S)$, and the constant in Lemma \ref{lemglobaltwoscale}. Choose $R$ so large that $Ch_{\cF}/R\le\eta/2$. Having chosen $R$, choose $1=\tau_0>\tau_1>\cdots>\tau_R>0$ recursively so that $C\tau_{j+1}\tau_j^{-h_{\cF}}\le\eta/(2R)$ for every $0\le j<R$. This is possible because each $\tau_j$ is positive.

Fix $j,F,S$ with $\rho_F(S)=\alpha_F$ and $B_S$ non-balanced. Fix an embedding $\psi:F[S]\hookrightarrow G[U_j]$. We count extensions $\phi:F\hookrightarrow G$ of $\psi$ with $\phi(V(F)\setminus S)\subseteq V(G)\setminus U_j$. Such an extension is called light if every vertex of $D_S$ is mapped outside $U_{j+1}$. For a light extension, every vertex of $D_S$ is mapped to a vertex of degree less than $\tau_{j+1}\Lambda$. Each vertex of $Z_S$ has at least one neighbour in $S$, because it is isolated in $F[V(F)\setminus S]$ and $F$ has no isolated vertices. Once $\psi$ is fixed, each vertex of $Z_S$ therefore has at most $\Lambda$ choices. The residual graph $G[V(G)\setminus U_j]$ has at most $m$ edges and maximum degree at most $\Lambda$. Lemma \ref{lemtwoscale}, applied to the non-balanced graph $B_S$, gives at most $C\tau_{j+1}\delta^{|V(B_S)|-\alpha^*(B_S)}\Lambda^{\alpha^*(B_S)}$ choices for the vertices of $B_S$ in light extensions. The number of possible embeddings $\psi$ is at most $|U_j|^{|S|}\le(2\delta/\tau_j)^{|S|}$, because $|U_j|\tau_j\Lambda\le\sum_{x\in U_j}d_G(x)\le2m=2\delta\Lambda$. Since $\rho_F(S)=|Z_S|+\alpha^*(B_S)=\alpha_F$, the number of light extensions for this fixed $F,S,j$ is at most $C\tau_{j+1}\tau_j^{-h_{\cF}}A_F$. After summing over all relevant pairs $(F,S)$ and all weights, the total light contribution for this fixed $j$ is at most $\eta A_{\cF}/(2R)$.

An extension that is not light maps at least one vertex of $D_S$ into $U_{j+1}$. Since every vertex of $D_S$ belongs to $V(F)\setminus S$ and the extension is counted with $S_j(\phi)=S$, this vertex is not in $S_j(\phi)$ and is in $S_{j+1}(\phi)$. Hence $S_{j+1}(\phi)$ properly contains $S_j(\phi)$. For a fixed embedding $\phi:F\hookrightarrow G$, such proper containments can occur for at most $|V(F)|\le h_{\cF}$ values of $j$. Therefore the sum over $j=0,\ldots,R-1$ of the total weighted number of non-light non-admissible embeddings is at most $h_{\cF}\sum_{F\in\cF}a_F\Emb(F,G)$. By Lemma \ref{lemglobaltwoscale}, after increasing $C$ if necessary, this sum is at most $Ch_{\cF}A_{\cF}$. Hence some $j\in\{0,\ldots,R-1\}$ has non-light contribution at most $Ch_{\cF}A_{\cF}/R\le\eta A_{\cF}/2$. For this same $j$, the light contribution is at most $\eta A_{\cF}/(2R)\le\eta A_{\cF}/2$. This proves the lemma.
\end{proof}

We now bound the admissible embeddings from a fixed synchronized core by common-neighbourhood data. This expression will later be realized by a threshold graph.

Let $G$ be a graph, let $U\subseteq V(G)$, and let $X=V(G)\setminus U$. For $A\subseteq U$, define $P_A=\{x\in X\mid N_G(x)\cap U=A\}$ and $y_A=|P_A|/\Lambda$. For $\Gamma\subseteq U$, define
\[
Y_\Gamma=\sum_{A\supseteq\Gamma}y_A=\frac1\Lambda\left|X\cap\bigcap_{u\in\Gamma}N_G(u)\right|.
\]
Let $c=e(G[X])/m$. For an admissible state $S$ of $F$ and an embedding $\psi:F[S]\hookrightarrow G[U]$, write $\Gamma_\psi(z)=\psi(N_F(z)\cap S)$ for $z\in Z_S$. This set is nonempty, since $z$ is isolated in $F[V(F)\setminus S]$ and $F$ has no isolated vertices. Define
\[
\mathsf{Adm}_F(G,U)=\sum_{\substack{S\subseteq V(F)\\ S\text{ admissible}}}\ \sum_{\psi:F[S]\hookrightarrow G[U]}\Lambda^{|Z_S|}(2cm)^{|V(B_S)|/2}\prod_{z\in Z_S}Y_{\Gamma_\psi(z)}.
\]
The factor $(2cm)^0$ is interpreted as $1$.

\begin{lemma}\label{lemadmissibleupper}
For a fixed set $U\subseteq V(G)$, the number of embeddings $\phi:F\hookrightarrow G$ for which $S=\phi^{-1}(U)$ is admissible is at most $\mathsf{Adm}_F(G,U)$.
\end{lemma}

\begin{proof}
Fix an admissible state $S$ and an embedding $\psi:F[S]\hookrightarrow G[U]$. Each vertex $z\in Z_S$ must be mapped into $X\cap\bigcap_{u\in\Gamma_\psi(z)}N_G(u)$, which has size $\Lambda Y_{\Gamma_\psi(z)}$. Hence the vertices of $Z_S$ have at most $\prod_{z\in Z_S}\Lambda Y_{\Gamma_\psi(z)}$ choices after collisions between these vertices are ignored. The graph $B_S$ is empty or balanced. The vertices of $B_S$ must be mapped into $X$, and Lemma \ref{lembalancedembedding} gives at most $(2e(G[X]))^{|V(B_S)|/2}=(2cm)^{|V(B_S)|/2}$ choices for them, with value $1$ when $B_S$ is empty. Ignoring injectivity between $Z_S$ and $V(B_S)$ and ignoring all additional edge constraints can only increase the count. Multiplying the estimates and summing over all admissible $S$ and all embeddings $\psi$ proves the lemma.
\end{proof}

The following lemma uses prescribed common-neighbourhood data to build a chained threshold graph, with error estimates uniform in the relevant parameters and with the dependence on $r=|U|$ made explicit.

\begin{lemma}\label{lemchainrealization}
Fix $K>0$. Let $G$ be a graph with $|V(G)|\le n$ and $e(G)\le m$, and let $U\subseteq V(G)$ satisfy $|U|\le K\delta$. For every fixed $0<\xi<1$, there is a threshold graph $T$ with $|V(T)|\le n$ and $e(T)\le m$ such that
\[
\sum_{F\in\cF}a_F\Emb(F,T)\ge(1-\xi)^{h_{\cF}}\sum_{F\in\cF}a_F\mathsf{Adm}_F(G,U)-o_{K,\xi,\cF}(A_{\cF}).
\]
The error term is uniform over all graphs $G$ and all sets $U$ satisfying the displayed hypotheses.
\end{lemma}

\begin{proof}
Let $r=|U|$. If $r>0$, write $p_u=Y_{\{u\}}$ for $u\in U$, and order $U=\{u_1,\ldots,u_r\}$ so that $p_{u_1}\ge\cdots\ge p_{u_r}$. Let $p_{u_{r+1}}=0$. If $r=0$, all sums involving the ordered set $U$ and all sets $P_i$ below are empty. For every nonempty $\Gamma\subseteq U$, the set counted by $Y_\Gamma$ is contained in the set counted by $Y_{\{u\}}$ for every $u\in\Gamma$, and hence $Y_\Gamma\le\min_{u\in\Gamma}p_u$. For $1\le i\le r$, define the layer-size parameter $p_{u_i}-p_{u_{i+1}}$ and the corresponding core neighbourhood $\{u_1,\ldots,u_i\}$. In this chained system the common-neighbourhood value is $Y'_\Gamma=\min_{u\in\Gamma}p_u$ for every nonempty $\Gamma\subseteq U$, and therefore $Y'_\Gamma\ge Y_\Gamma$.

We now realize the chained data by a threshold graph. Let $Q_0$ be a clique of size $t$ satisfying $\binom t2\le(1-\xi)cm<\binom{t+1}2$ when $c>0$, and let $Q_0=\emptyset$ when $c=0$. For $1\le i\le r$, let $P_i$ be a set of size $\floor{(1-\xi)\Lambda(p_{u_i}-p_{u_{i+1}})}$. Let $U^T=\{u_1^T,\ldots,u_r^T\}$ be disjoint from $Q_0\cup P_1\cup\cdots\cup P_r$, and let $\iota:U\to U^T$ be the bijection given by $\iota(u_i)=u_i^T$. Construct a graph by first creating the clique $Q_0$, then adding the vertices of $P_r$ as isolated vertices and then adding the vertex $u_r^T$ as a dominating vertex, then adding the vertices of $P_{r-1}$ as isolated vertices and then adding the vertex $u_{r-1}^T$ as a dominating vertex, and continuing in this way until $P_1$ and $u_1^T$ have been added. This graph is threshold by construction. In the final graph, the core $U^T$ is a clique, each $P_i$ is adjacent exactly to $u_1^T,\ldots,u_i^T$ among the core vertices, every core vertex is adjacent to every vertex of $Q_0$, there are no edges from any $P_i$ to $Q_0$, and there are no edges inside or between the sets $P_i$.

We check the edge bound. The number of edges with one endpoint in $U^T$ and the other endpoint in $P_1\cup\cdots\cup P_r$ is
\[
\sum_{i=1}^r i|P_i|
\le
(1-\xi)\Lambda\sum_{i=1}^r i(p_{u_i}-p_{u_{i+1}})
=
(1-\xi)\Lambda\sum_{i=1}^r p_{u_i},
\]
where $p_{u_{r+1}}=0$ and the last equality follows by telescoping. Since $\Lambda\sum_{i=1}^rp_{u_i}=e_G(U,X)$, these edges contribute at most $(1-\xi)e_G(U,X)$. The clique $Q_0$ contributes at most $(1-\xi)cm=(1-\xi)e(G[X])$ edges. The edges inside $U^T$ contribute $O_K(r^2)$ edges. The edges with one endpoint in $U^T$ and the other endpoint in $Q_0$ contribute $O_K(rQ)$ edges, since $|Q_0|=t=O_K(Q)$. Hence
\[
e(T)\le (1-\xi)e_G(U,X)+(1-\xi)e(G[X])+O_K(r^2+rQ).
\]
By Lemma~\ref{lemscaleseparation} and the bound $r\le K\delta$, we have $r^2=o_K(m)$ and $rQ=o_K(m)$. Therefore
\[
e(T)\le(1-\xi)(e_G(U,X)+e(G[X]))+o_K(m)\le(1-\xi)m+o_K(m)<m
\]
for all sufficiently large $n$, uniformly over all such $G$ and $U$.

We check the vertex bound. If $r>0$, the total number of vertices in $P_1\cup\cdots\cup P_r$ is at most $(1-\xi)\Lambda p_{u_1}+O_K(r)\le(1-\xi)\Lambda+O_K(r)$, because $p_{u_1}\le1$. If $r=0$, this number is $0$. The core has $r\le K\delta=o_K(\Lambda)$ vertices, and $Q_0$ has $O(Q)=o(\Lambda)$ vertices by Lemma \ref{lemscaleseparation}. Thus $|V(T)|\le(1-\xi)\Lambda+o_K(\Lambda)\le\Lambda\le n$ for all sufficiently large $n$, uniformly over all such $G$ and $U$.

It remains to compare embedding counts. Fix $F\in\cF$, an admissible state $S$, and an embedding $\psi:F[S]\hookrightarrow G[U]$. The map $\iota\circ\psi$ sends $F[S]$ into $U^T$, and it is edge preserving because $U^T$ is a clique. For each $z\in Z_S$, the set $\Gamma_\psi(z)$ is nonempty. The common neighbourhood of $\iota(\Gamma_\psi(z))$ inside $P_1\cup\cdots\cup P_r$ has size at least $(1-\xi)\Lambda Y'_{\Gamma_\psi(z)}-r$, because the floor operation changes each of the sizes $|P_i|$ by less than one. Since $Y'_{\Gamma_\psi(z)}\ge Y_{\Gamma_\psi(z)}$, this size is at least $(1-\xi)\Lambda Y_{\Gamma_\psi(z)}-r$. The number of injective choices for all vertices of $Z_S$ is at least
\[
(1-\xi)^{|Z_S|}\Lambda^{|Z_S|}\prod_{z\in Z_S}Y_{\Gamma_\psi(z)}-C_F(r+1)\Lambda^{|Z_S|-1},
\]
where the error term is interpreted as $0$ when $Z_S=\emptyset$. To prove this estimate, let $z_1,\ldots,z_a$ be the vertices of $Z_S$, and let $M_i$ be the available set for $z_i$ inside $P_1\cup\cdots\cup P_r$. The preceding paragraph gives $|M_i|\ge (1-\xi)\Lambda Y_{\Gamma_\psi(z_i)}-r$. Since $0\le Y_{\Gamma_\psi(z_i)}\le1$, the product $\prod_i|M_i|$ is at least $(1-\xi)^a\Lambda^a\prod_iY_{\Gamma_\psi(z_i)}-C_F(r+1)\Lambda^{a-1}$. At most $C_F\Lambda^{a-1}$ choices have a collision among the images of $z_1,\ldots,z_a$. Removing these choices gives the displayed lower bound for injective choices.

The vertices of $B_S$ are chosen from the clique $Q_0$. Let $b=|V(B_S)|$. If $b=0$, this contributes $1$. If $b>0$, the choice of $t$ gives
\[
(t)_b\ge(1-\xi)^b(2cm)^{b/2}-C_FQ^{b-1}.
\]
This is true uniformly for $0\le c\le1$. When $cm=0$, the right hand side is nonpositive after increasing $C_F$. When $cm>0$, the inequalities $\binom t2\le(1-\xi)cm<\binom{t+1}2$ imply $t\ge c_\xi\sqrt{cm}-1$, and the fixed falling factorial expansion gives the displayed estimate after decreasing the main coefficient to $(1-\xi)^b$ and increasing $C_F$.

All choices described above give labelled embeddings of $F$ in $T$. Edges inside $S$ are present because $U^T$ is a clique. Edges from $Z_S$ to $S$ are present by the chained common-neighbourhood condition and the map $\iota\circ\psi$ on $S$. There are no edges from $Z_S$ to $Z_S\cup V(B_S)$ in $F$, because $Z_S$ is isolated in $F[V(F)\setminus S]$. Edges inside $B_S$ are present because $Q_0$ is a clique. Edges from $B_S$ to $S$ are present because $U^T$ is joined to $Q_0$.

The embeddings counted for different pairs $(S,\psi)$ are disjoint. To see this, every embedding counted for a fixed pair maps exactly the vertices of $S$ into $U^T$, maps the vertices of $Z_S$ into $P_1\cup\cdots\cup P_r$, and maps the vertices of $B_S$ into $Q_0$. Thus the preimage of $U^T$ is exactly $S$. Once an embedding is fixed, its restriction to this preimage followed by $\iota^{-1}$ is exactly $\psi$. Therefore the same labelled embedding cannot be counted under a different state or a different core embedding.

For the fixed $F,S,\psi$, multiplying the estimates for $Z_S$ and $B_S$ gives at least
\[
(1-\xi)^{|Z_S|+|V(B_S)|}\Lambda^{|Z_S|}(2cm)^{|V(B_S)|/2}\prod_{z\in Z_S}Y_{\Gamma_\psi(z)}-E_{F,S,\psi},
\]
where
\[
E_{F,S,\psi}\le C_F(r+1)\Lambda^{|Z_S|-1}Q^{|V(B_S)|}+C_F\Lambda^{|Z_S|}Q^{|V(B_S)|-1}.
\]
The second term is omitted when $|V(B_S)|=0$. Since $S$ is admissible, one has $\alpha_F=|Z_S|+|V(B_S)|/2$ and $v_F-\alpha_F=|S|+|V(B_S)|/2$. The number of possible maps $\psi:F[S]\hookrightarrow G[U]$ is at most $r^{|S|}$. Summing the first error term over all such $\psi$ gives at most
\[
C_Fr^{|S|}(r+1)\Lambda^{|Z_S|-1}Q^{|V(B_S)|}.
\]
Dividing this by $A_F=\delta^{v_F-\alpha_F}\Lambda^{\alpha_F}$ gives at most
\[
C_F\frac{r^{|S|}(r+1)}{\delta^{|S|}\Lambda}.
\]
Since $r\le K\delta$ and $\delta/\Lambda\to0$, this ratio is $o_{K,F}(1)$. Summing the second error term over all such $\psi$ gives at most $C_Fr^{|S|}\Lambda^{|Z_S|}Q^{|V(B_S)|-1}$. Dividing by $A_F$ gives at most $C_FK^{|S|}/Q$, which is $o_{K,F}(1)$ because $Q\to\infty$. Hence the total error after summing over all $\psi$ is $o_{K,F}(A_F)$.

Since $0<1-\xi<1$ and $|S|+|Z_S|+|V(B_S)|=|V(F)|\le h_{\cF}$, we have $(1-\xi)^{|Z_S|+|V(B_S)|}\ge(1-\xi)^{h_{\cF}}$. For each fixed $F$, the embeddings produced from distinct pairs $(S,\psi)$ are disjoint, since the state $S$ and the core restriction $\psi$ are determined by the embedding itself. Therefore the corresponding lower bounds may be summed over all admissible states $S$, all core embeddings $\psi$, and then over all $F\in\cF$ with weights $a_F$. The family $\cF$ and all state sets are finite, so the sum of the uniform errors is $o_{K,\xi,\cF}(A_{\cF})$. Therefore
\[
\sum_{F\in\cF}a_F\Emb(F,T)\ge(1-\xi)^{h_{\cF}}\sum_{F\in\cF}a_F\mathsf{Adm}_F(G,U)-o_{K,\xi,\cF}(A_{\cF}).
\]
This proves the lemma.
\end{proof}

We now prove Theorem~\ref{thmweightedthreshold}. The proof starts with an arbitrary near-extremal graph $G$, chooses a synchronized core $U$, bounds the non-admissible states as lower-order contributions, and then applies Lemma~\ref{lemchainrealization} to realize the admissible contribution in a threshold graph.

\begin{proof}[Proof of Theorem~\ref{thmweightedthreshold}]
Let
\[
M_{\cF}=\max_{\substack{|V(G)|\le n\\ e(G)\le m}}\sum_{F\in\cF}a_F\Emb(F,G)
\]
and let $M^\cT_{\cF}$ be the corresponding maximum restricted to threshold graphs. The inequality $M^\cT_{\cF}\le M_{\cF}$ is immediate. We prove the reverse asymptotic inequality.

Fix $\eta>0$, and choose a graph $G$ attaining $M_{\cF}$. Apply Lemma \ref{lemlayersynchronization} to $G$, and let $U=U_j$ be the synchronized core. Since $|U_j|\le2\delta/\tau_j$ and $\tau_j$ is fixed once $\eta$ and $\cF$ are fixed, there is a constant $K=K(\eta,\cF)$ such that $|U|\le K\delta$.

We first count embeddings whose state has $\rho_F(S)<\alpha_F$. Fix $F$ and $S\subseteq V(F)$ with $\rho_F(S)<\alpha_F$. The number of possible images of the vertices in $S$ inside $U$ is at most $|U|^{|S|}=O_{\eta,\cF}(\delta^{|S|})$. Once these images are fixed, each vertex of $Z_S$ has at most $\Lambda$ choices, because it has at least one neighbour in $S$ unless $Z_S$ is empty. The vertices of $B_S$ are mapped into $V(G)\setminus U$, and Lemma \ref{lemtwoscale} gives at most $C_F\delta^{|V(B_S)|-\alpha^*(B_S)}\Lambda^{\alpha^*(B_S)}$ choices for them. Therefore this state contributes at most
\[
C_{\eta,\cF}\delta^{|S|+|V(B_S)|-\alpha^*(B_S)}\Lambda^{|Z_S|+\alpha^*(B_S)}.
\]
Since $v_F=|S|+|Z_S|+|V(B_S)|$ and $\rho_F(S)=|Z_S|+\alpha^*(B_S)$, division by $A_F=\delta^{v_F-\alpha_F}\Lambda^{\alpha_F}$ gives at most $C_{\eta,\cF}(\delta/\Lambda)^{\alpha_F-\rho_F(S)}$. Lemma \ref{lemscaleseparation} gives $\delta/\Lambda\to0$, and there are only finitely many $F$ and $S$. Hence all states with $\rho_F(S)<\alpha_F$ contribute $o_{\eta,\cF}(A_{\cF})$ in total.

By Lemma \ref{lemstatemonotonicity}, no state satisfies $\rho_F(S)>\alpha_F$. By the choice of $U$, together with Lemma \ref{lemlayersynchronization} we have that all states with $\rho_F(S)=\alpha_F$ that are not admissible contribute at most $\eta A_{\cF}$. The remaining embeddings have admissible states, and Lemma \ref{lemadmissibleupper} gives
\[
M_{\cF}\le\sum_{F\in\cF}a_F\mathsf{Adm}_F(G,U)+\eta A_{\cF}+o_{\eta,\cF}(A_{\cF}).
\]
Fix $0<\xi<1$. By Lemma \ref{lemchainrealization}, applied with this value of $K$, there is a threshold graph $T$ satisfying $|V(T)|\le n$ and $e(T)\le m$ such that
\[
M^\cT_{\cF}\ge\sum_{F\in\cF}a_F\Emb(F,T)\ge(1-\xi)^{h_{\cF}}\sum_{F\in\cF}a_F\mathsf{Adm}_F(G,U)-o_{K,\xi,\cF}(A_{\cF}).
\]
Combining the two displayed inequalities gives
\[
M_{\cF}\le(1-\xi)^{-h_{\cF}}M^\cT_{\cF}+\eta A_{\cF}+o_{\eta,\xi,\cF}(A_{\cF}).
\]
By Lemma \ref{lemlowerscale} $M^\cT_{\cF}\ge c_{\cF}A_{\cF}$. Therefore
\[
M_{\cF}\le\left((1-\xi)^{-h_{\cF}}+\frac{\eta}{c_{\cF}}+o_{\eta,\xi,\cF}(1)\right)M^\cT_{\cF}.
\]
Taking first the limit superior as $n\to\infty$, then letting $\eta\rightarrow0$ and $\xi\rightarrow0$, gives $M_{\cF}\le(1+o(1))M^\cT_{\cF}$. Together with $M^\cT_{\cF}\le M_{\cF}$, this proves Theorem \ref{thmweightedthreshold}.
\end{proof}

We finally derive the homomorphism statement from the embedding statement through the finite quotient decomposition.

\begin{proof}[Proof of Theorem \ref{thmbp}]
Assume first that $H$ has no isolated vertices. By Lemma \ref{lemquotient}, there are finitely many graphs $J\in\cQ(H)$ and nonnegative integers $a_J$ such that $\homg(H,G)=\sum_{J\in\cQ(H)}a_J\Emb(J,G)$ for every simple graph $G$. Every quotient with $a_J>0$ has no isolated vertices. Theorem \ref{thmweightedthreshold}, applied to this finite weighted family, gives
\[
\max_{\substack{|V(G)|\le n\\ e(G)\le m}}\homg(H,G) = (1+o(1)) \max_{\substack{|V(T)|\le n\\ e(T)\le m\\ T\in\cT}}\homg(H,T).
\]
This proves the theorem when $H$ has no isolated vertices.

For a general graph $H$, write $H=H_0\sqcup rK_1$, where $H_0$ has no isolated vertices and $r\ge0$. If $H_0$ is empty, then $\homg(H,G)=|V(G)|^r\le n^r$, and equality is attained by the edgeless threshold graph on $n$ vertices. Suppose that $H_0$ is nonempty. Let $G^\star$ be a graph attaining $\cM(H,n,m)$. Then
\[
\cM(H,n,m)=|V(G^\star)|^r\homg(H_0,G^\star)\le n^r\max_{\substack{|V(G)|\le n\\ e(G)\le m}}\homg(H_0,G).
\]
By the case already proved for graphs without isolated vertices, there is a threshold graph $T_0$ with $|V(T_0)|\le n$ and $e(T_0)\le m$ such that
\[
\homg(H_0,T_0)\ge(1-o(1))\max_{\substack{|V(G)|\le n\\ e(G)\le m}}\homg(H_0,G).
\]
Add isolated vertices to $T_0$ until the graph has exactly $n$ vertices. The resulting graph $T$ is threshold and still has at most $m$ edges. Since $H_0$ has no isolated vertices, adding isolated target vertices does not change $\homg(H_0,\cdot)$. Hence
\[
\max_{\substack{|V(T')|\le n\\ e(T')\le m\\ T'\in\cT}}\homg(H,T')\ge\homg(H,T)=n^r\homg(H_0,T_0)\ge(1-o(1))\cM(H,n,m).
\]
The reverse inequality is immediate since threshold graphs are allowed host graphs. This completes the proof.
\end{proof}

\section{Proof of Theorem \ref{thmbipmain}}\label{secbipproof}

Throughout this section, $H$ is a fixed bipartite graph with no isolated vertices. Let $h=|V(H)|$ and let $g=|E(H)|$. We use the notation from Section \ref{secintro}. Thus $H_1,\ldots,H_r$ are the connected components of $H$, each $H_i$ has a fixed bipartition $A_i\cup B_i$ with $|A_i|\le |B_i|$, $\sigma(H)=\sum_i|A_i|$, and $\Delta(H)=\max_i\max_{X\subseteq A_i}(|X|-|N_{H_i}(X)|)$. Since labelled and unlabelled copies differ by the constant factor $|\Aut(H)|$, it is enough to prove the estimates for labelled embeddings.

We first prove the order of $H$ in the quasi-complete bipartite graph. This is the scale against which the finite kernel correction is compared.

\begin{lemma}\label{lemqcbscale}
Assume that $m/n\to\infty$ and $n\le m\le n^2/4$. Then
\[
M(B_n^m,H)=\Theta_H\left(m^{\sigma(H)}n^{h-2\sigma(H)}\right).
\]
Consequently $N(B_n^m,H)=\Theta_H(m^{\sigma(H)}n^{h-2\sigma(H)})$.
\end{lemma}

\begin{proof}
Let $t$ be the smallest positive integer such that $t(n-t)\ge m$, and set $s=n-t$. Since $m\le n^2/4$, the smaller real root $x$ of $x(n-x)=m$ satisfies $x\le n/2$. The minimality of $t$ gives $t\le x+1$. Since $x=2m/(n+\sqrt{n^2-4m})$, one has $x\le2m/n$. Hence $t\le 2m/n+1\le3m/n$ for all sufficiently large $n$. Also $t\ge m/n$, since $t(n-t)\le tn$. Thus $t=\Theta(m/n)$. Since $t\le n/2+1$, one has $s=n-t=\Theta(n)$.


The graph $B_n^m$ is a subgraph of $K_{t,s}$. Fix a component $H_i$. Since $H_i$ is connected, an embedding of $H_i$ into $K_{t,s}$ sends one colour class of $H_i$ to the part of size $t$ and the other colour class to the part of size $s$. Hence the number of labelled embeddings of $H_i$ into $K_{t,s}$ is at most
\[
C_{H_i}\left(t^{|A_i|}s^{|B_i|}+t^{|B_i|}s^{|A_i|}\right).
\]
Since $t\le s+1$ and $|A_i|\le |B_i|$, this is at most $C'_{H_i}t^{|A_i|}s^{|B_i|}$. Multiplying over the components and ignoring collisions between different components can only increase the count. Therefore $M(B_n^m,H)\le C_Ht^{\sigma(H)}s^{h-\sigma(H)}=O_H(m^{\sigma(H)}n^{h-2\sigma(H)})$.

For the lower bound, all deleted edges are incident with one vertex in the part of size $t$, so $B_n^m$ contains a complete bipartite graph with parts of sizes $t-1$ and $s$. Since $t\to\infty$ and $s\to\infty$, the number of labelled embeddings obtained by sending each $A_i$ into the part of size $t-1$ and each $B_i$ into the part of size $s$ is at least $c_H(t-1)^{\sigma(H)}s^{h-\sigma(H)}$. This is $\Omega_H(m^{\sigma(H)}n^{h-2\sigma(H)})$. The statement for unlabelled copies follows from $M(G,H)=|\Aut(H)|N(G,H)$.
\end{proof}

We next prove the finite kernel order correction. The upper bound is a weighted Shearer projection bound, and the lower bound is a rounded blow up of the continuous kernel.

\begin{lemma}\label{lembipkernelupper}
For every bipartite graph $G$ with $|V(G)|=n$ and $e(G)=m$, where $n\le m\le n^2$, one has $M(G,H)\le2^h\kappa_H(n,m)$.
\end{lemma}

\begin{proof}
Let $\Omega=\Emb(H,G)$. Choose nonnegative numbers $\mu_v$ and $\lambda_{uv}$ attaining the dual optimum in Lemma \ref{lemkappalp}. The feasible dual choice $\mu_v=1$ for every $v\in V(H)$ and $\lambda_{uv}=0$ for every $uv\in E(H)$ has value $h\log n$. Hence the optimal dual value is at most $h\log n$. It follows that $\log m\sum_{uv\in E(H)}\lambda_{uv}\le h\log n$, and therefore $\sum_{uv\in E(H)}\lambda_{uv}\le h$.

For each vertex $v\in V(H)$, define a projection $\pi_v(\Omega)=\{\phi(v):\phi\in \Omega\}$. For each edge $uv\in E(H)$, define a projection $\pi_{uv}(\Omega)=\{(\phi(u),\phi(v)):\phi\in \Omega\}$. It simply follows that $|\pi_v(\Omega)|\leq n$ for $v\in V(H)$ and $|\pi_{uv}(\Omega)|\leq 2m$ for $uv \in E(H)$. The dual covering condition $\mu_v+\sum_{uv\in E(H)}\lambda_{uv}\ge1$ for every $v\in V(H)$ from Lemma~\ref{lemkappalp} is exactly the hypothesis of Lemma \ref{lemshearer} for these projections. 
Therefore
\[
M(G,H)=|\Omega|\le\prod_{v\in V(H)}n^{\mu_v}\prod_{uv\in E(H)}(2m)^{\lambda_{uv}}.
\]
Using $\sum_{uv}\lambda_{uv}\le h$ and Lemma \ref{lemkappalp}, we obtain
\[
M(G,H)\le2^h\exp\left(\log n\sum_v\mu_v+\log m\sum_{uv}\lambda_{uv}\right)=2^h\kappa_H(n,m).
\]
This proves the lemma.
\end{proof}

The next lemma realizes the continuous kernel by an actual bipartite graph. The scaling and rounding are included so that the host has exactly $n$ vertices and exactly $m$ edges.

\begin{lemma}\label{lembipkernellower}
There are constants $c_H>0$ and $n_H$ such that, for all integers $n,m$ with $n\ge n_H$ and $n\le m\le\floor{n^2/4}$, one has $\Mbip(n,m,H)\ge c_H\kappa_H(n,m)$.
\end{lemma}

\begin{proof}
Fix a bipartition $X_H\cup Y_H$ of $H$. Let $(a_v)_{v\in V(H)}$ be a maximizer in the definition of $\kappa_H(n,m)$. Recall that $h=|V(H)|$ and $g=|E(H)|$. Choose a constant $\eta>0$, depending only on $H$, such that $\eta\le1$, $g\eta\le1/8$, and $h\eta\le1/16$. For every $v\in V(H)$, define $b_v=\max\{1,\floor{\eta a_v}\}$. If $\eta a_v\ge2$, then $b_v\ge\eta a_v/2$. If $\eta a_v<2$, then $b_v=1\ge\eta a_v/2$. Hence
\[
\prod_{v\in V(H)}b_v\ge \left(\frac{\eta}{2}\right)^h\prod_{v\in V(H)}a_v =\left(\frac{\eta}{2}\right)^h\kappa_H(n,m).
\]

We claim that $b_ub_v\le\eta m$ for every edge $uv\in E(H)$ once $n$ is sufficiently large. If $\eta a_u<1$ and $\eta a_v<1$, then $b_ub_v=1\le\eta m$, since $m\ge n$ and $n$ is sufficiently large. If $\eta a_u<1\le\eta a_v$, then $b_ub_v\le\eta a_v\le\eta a_ua_v\le\eta m$, since $a_u\ge1$. The case $\eta a_v<1\le\eta a_u$ is identical. If $\eta a_u\ge1$ and $\eta a_v\ge1$, then $b_ub_v\le\eta^2a_ua_v\le\eta m$. This proves the claim.

Construct a blow up of $H$ with pairwise disjoint vertex classes $V_v$, where $|V_v|=b_v$. Place $V_v$ in the left host side if $v\in X_H$, and place $V_v$ in the right host side if $v\in Y_H$. Add all edges between $V_u$ and $V_v$ whenever $uv\in E(H)$. By the claim, the number of edges is at most $\sum_{uv\in E(H)}b_ub_v\leq m\eta |E(H)|\le m/8$. The number of used vertices in either side is at most $\sum_v(\eta a_v+1)\le h\eta n+h\le n/4$ for all sufficiently large $n$. Place this blow up inside a balanced bipartite graph on $n$ vertices, with parts of sizes $\floor{n/2}$ and $\ceil{n/2}$. Since $m\le\floor{n^2/4}$, the complete bipartite graph between the two host parts has at least $m$ edges. Add missing cross edges until the total edge number is exactly $m$. The resulting graph is bipartite, has $n$ vertices, and has exactly $m$ edges.

Every choice of one vertex from each class $V_v$ gives an injective labelled embedding of $H$, because the classes are pairwise disjoint and every edge class prescribed by $H$ is complete. Hence this graph has at least $\prod_vb_v\ge(\eta/2)^h\kappa_H(n,m)$ labelled embeddings of $H$. Taking $c_H=(\eta/2)^h$ proves the lemma.
\end{proof}

We now compare the kernel scale with the quasi-complete scale under the Hall matching condition. This proves the order-level positive part of the quasi-complete criterion.

\begin{lemma}\label{lemhallpositive}
Assume that $\Delta(H)=0$ and let $q_H(n,m)=m^{\sigma(H)}n^{h-2\sigma(H)}$. If $n\le m\le n^2$, then
$\kappa_H(n,m)=m^{\sigma(H)}n^{h-2\sigma(H)}.$
\end{lemma}

\begin{proof}
Since $\Delta(H)=0$, Hall's theorem gives, for every component $H_i$, a matching from $A_i$ into $B_i$ that saturates $A_i$. Let $(a_v)_{v\in V(H)}$ be any feasible vector in the definition of $\kappa_H(n,m)$. For each matched edge $uv$, one has $a_ua_v\le m$. Every vertex $v$ of $B_i$ not covered by the matching satisfies $a_v\le n$. Multiplying these inequalities over all components gives
\[
\prod_{v\in V(H)}a_v\le m^{\sigma(H)}n^{h-2\sigma(H)}=q_H(n,m).
\]
Taking the maximum over all feasible vectors gives $\kappa_H(n,m)\le q_H(n,m)$.

For the lower bound, choose $a_v=m/n$ for every $v\in A_i$ and choose $a_v=n$ for every $v\in B_i$. Since $n\le m\le n^2$, these choices satisfy $1\le a_v\le n$. Every edge product is exactly $m$. Hence this vector is feasible, and it gives $\kappa_H(n,m)\ge q_H(n,m)$. Combining the two bounds proves the lemma.
\end{proof}

We next construct a finite kernel that is larger than the quasi-complete scale when the smaller colour class has positive Hall deficiency. This proves the order-level negative part of the quasi-complete criterion.

\begin{lemma}\label{lemhalldeficiency}
Assume that $m/n\to\infty$ and $m/n^2\to0$. If $d=\Delta(H)>0$, then there is a constant $c_H>0$ such that, for all sufficiently large $n$,
\[
\Mbip(n,m,H)\ge c_Hm^{\sigma(H)}n^{h-2\sigma(H)}\left(\frac{n^2}{m}\right)^d.
\]
\end{lemma}

\begin{proof}
Choose a component $H_1$ and a set $A_0\subseteq A_1$ such that $|A_0|-|N_{H_1}(A_0)|=d$. Let $N_0=N_{H_1}(A_0)$. Since $m/n^2\to0$, we may assume throughout the proof that $m\le n^2$. Recall that $h=|V(H)|$ and $g=|E(H)|$. Choose a constant $\eta>0$, depending only on $H$, such that $g\eta^2\le1/8$ and $h\eta\le1/16$. For every vertex $v\in V(H)$, define a class size $m_v$ as follows. In the component $H_1$, let $m_v=\floor{\eta n}$ for $v\in A_0\cup(B_1\setminus N_0)$, and let $m_v=\floor{\eta m/n}$ for $v\in N_0\cup(A_1\setminus A_0)$. In every other component $H_i$ for $i\ge 2$, let $m_v=\floor{\eta m/n}$ for $v\in A_i$, and let $m_v=\floor{\eta n}$ for $v\in B_i$. Since $m/n\to\infty$, all numbers $m_v$ tend to infinity.

Choose pairwise disjoint vertex classes $V_v$ with $|V_v|=m_v$. Place $V_v$ in the left host side if $v\in A_i$ for some $i$, and place $V_v$ in the right host side if $v\in B_i$ for some $i$. Add all edges between $V_u$ and $V_v$ whenever $uv\in E(H)$. We check the edge count. If an edge has one endpoint in $A_0$, then its other endpoint lies in $N_0$ by the definition of $N_0$, so the two class sizes are at most $\eta n$ and $\eta m/n$. If an edge has one endpoint in $A_1\setminus A_0$, then its class has size at most $\eta m/n$, while the other endpoint has class size at most $\eta n$ since $m\le n^2$. In every component other than $H_1$, every edge has one endpoint of scale $m/n$ and one endpoint of scale $n$. Therefore each edge class has size at most $\eta^2m$. Hence the blow up has at most $g\eta^2m\le m/8$ edges.

The number of used vertices in either side of the host is at most $h\eta n+h\le n/4$ for all sufficiently large $n$, since $m/n\le n$. Place the construction inside a balanced bipartite graph on $n$ vertices. Since $m/n^2\to 0$, we have $m\le\floor{n^2/4}$ for all sufficiently large $n$. Add arbitrary missing cross edges until the total number of edges is exactly $m$. Adding edges cannot destroy existing embeddings.

Every choice of one vertex from each class $V_v$ gives a labelled embedding of $H$. For all sufficiently large $n$, each class of scale $n$ has size at least $\eta n/2$, and each class of scale $m/n$ has size at least $\eta m/(2n)$. Hence
\[
\Mbip(n,m,H)\ge c_Hn^{|A_0|+|B_1\setminus N_0|}\left(\frac {m}{n}\right)^{|N_0|+|A_1\setminus A_0|}\prod_{i=2}^r\left(\frac {m}{n}\right)^{|A_i|}n^{|B_i|}.
\]
The quasi-complete scale is $q_H(n,m)=m^{\sigma(H)}n^{h-2\sigma(H)}=\prod_{i=1}^r(m/n)^{|A_i|}n^{|B_i|}$. Dividing the preceding lower bound by $q_H(n,m)$ gives
\[
\frac{\Mbip(n,m,H)}{q_H(n,m)}\ge c_Hn^{|A_0|-|N_0|}\left(\frac {m}{n}\right)^{|N_0|-|A_0|}=c_H\left(\frac{n^2}{m}\right)^d.
\]
This proves the lemma.
\end{proof}

We now combine the finite kernel estimates and the comparison with the quasi-complete scale.

\begin{proof}[Proof of Theorem \ref{thmbipmain}]
The upper bound $\Mbip(n,m,H)\le C_H\kappa_H(n,m)$ follows from Lemma \ref{lembipkernelupper} with $C_H=2^h$. The lower bound $\Mbip(n,m,H)\ge c_H\kappa_H(n,m)$ follows from Lemma \ref{lembipkernellower}. Since $M(G,H)=|\Aut(H)|N(G,H)$ for every host graph $G$, the statement $\Nbip(n,m,H)=\Theta_H(\kappa_H(n,m))$ follows. This proves part (i).

Assume now that $m/n\to\infty$ and $m/n^2\to0$. Then $m\le n^2/4$ for all sufficiently large $n$. Let $q_H(n,m)=m^{\sigma(H)}n^{h-2\sigma(H)}$. Lemma \ref{lemqcbscale} gives $N(B_n^m,H)=\Theta_H(q_H(n,m))$.

If $\Delta(H)=0$, then Lemma \ref{lemhallpositive} gives $\kappa_H(n,m)=\Theta_H(q_H(n,m))$. Part (i) gives $\Nbip(n,m,H)=\Theta_H(\kappa_H(n,m))$. Therefore $\Nbip(n,m,H)=\Theta_H(N(B_n^m,H))$. This proves part (ii).

If $\Delta(H)>0$, let $d=\Delta(H)$. By Lemma \ref{lemhalldeficiency} we have
\[
\Mbip(n,m,H)\ge c_Hq_H(n,m)\left(\frac{n^2}{m}\right)^d.
\]
Since $\Nbip(n,m,H)=\Mbip(n,m,H)/|\Aut(H)|$, after decreasing $c_H$ if necessary we have
\[
\Nbip(n,m,H)\ge c_Hq_H(n,m)\left(\frac{n^2}{m}\right)^d.
\]
Since $N(B_n^m,H)=\Theta_H(q_H(n,m))$, after decreasing $c_H$ if necessary we get
\[
\Nbip(n,m,H)\ge c_HN(B_n^m,H)\left(\frac{n^2}{m}\right)^d.
\]
Since $m/n^2\to0$, the factor $(n^2/m)^d$ tends to infinity. This proves part (iii) and completes the proof of Theorem \ref{thmbipmain}.
\end{proof}

\section{Examples and concluding remarks}\label{secexamples}

We close with examples illustrating the three finite kernel phenomena appearing in the paper. The first group evaluates the threshold polynomial $P_H(q)$ for small graphs. The second example shows the three active scales $\delta$, $\Lambda$, and $Q$ in the finite sparse threshold theorem. The last group compares Hall saturated and Hall deficient bipartite graphs.

\subsection{Small threshold constants}

For $K_2$, the half integral optimal vectors are $(1,0)$, $(0,1)$, and $(1/2,1/2)$. Hence $P_{K_2}(q)=2(1-q^2)/2+q^2=1$, so $C_T(K_2)=1$. Theorem \ref{thmds} gives $\cM_{K_2}(\beta)=\beta$, as it should.

For the complete graph $K_r$ with $r\ge3$, the unique vector in $\Phi^*(K_r)$ assigns value $1/2$ to every vertex. Thus $\alpha^*(K_r)=r/2$, $P_{K_r}(q)=q^r$, and $C_T(K_r)=1$. The sparse threshold asymptotic is $\cM_{K_r}(\beta)=\beta^{r/2}(1+o(1))$. The maximizing value is $q=1$, corresponding to the lower clique part of the three-step threshold graphon.

For the star $K_{1,r}$ with $r\ge2$, the unique vector in $\Phi^*(K_{1,r})$ assigns value $0$ to the centre and value $1$ to each leaf. Hence $\alpha^*(K_{1,r})=r$, $P_{K_{1,r}}(q)=(1-q^2)/2$, and $C_T(K_{1,r})=1/2$. The sparse threshold asymptotic is $\cM_{K_{1,r}}(\beta)=\beta(1/2+o(1))$. The maximizing value is $q=0$, corresponding to the upper dominating part of the threshold graphon.

For the cycle $C_4$, the vectors in $\Phi^*(C_4)$ are the two integral maximum independent set vectors and the all $1/2$ vector. Therefore $P_{C_4}(q)=2((1-q^2)/2)^2+q^4=1/2-q^2+3q^4/2$. The maximum over $0\le q\le1$ is $1$, attained at $q=1$. Hence Theorem \ref{thmds} gives $\cM_{C_4}(\beta)=\beta^2(1+o(1))$.

These examples show that different graphs select different parts of the same threshold kernel. Stars select the dominating part, complete graphs select the clique part, and $C_4$ contains both integral and half integral leading states even though the leading constant is maximized at the clique endpoint.

\subsection{A finite graph example with three active scales}

Let $F_\triangle$ be the graph with vertex set $\{d,i_1,i_2,r_1,r_2,r_3\}$ whose edges are $di_1$, $di_2$, the three edges $dr_j$ for $1\le j\le3$, and the three edges of the triangle on $\{r_1,r_2,r_3\}$. The vector with $x_d=0$, $x_{i_1}=x_{i_2}=1$, and $x_{r_j}=1/2$ for $1\le j\le3$ is feasible and has total weight $7/2$. Conversely, if $x$ is feasible, then the constraints on $di_1$ and $di_2$ give $x_{i_1}+x_{i_2}\le2(1-x_d)$, while the triangle constraints give $x_{r_1}+x_{r_2}+x_{r_3}\le3/2$. Hence $\sum_vx_v\le x_d+2(1-x_d)+3/2=7/2-x_d\le7/2$. Therefore $\alpha^*(F_\triangle)=7/2$, and since $v_{F_\triangle}=6$, we have $A_{F_\triangle}=\delta^{5/2}\Lambda^{7/2}$. The lower bound construction in Lemma \ref{lemlowerscale} has a core $C$ with $|C|=\Theta(\delta)$ for the vertex $d$, an independent set $P$ with $|P|=\Theta(\Lambda)$ for the vertices $i_1,i_2$, and a clique $W$ with $|W|=\Theta(Q)$ for the triangle vertices $r_1,r_2,r_3$. The edge count is governed by the complete bipartite graph between $C$ and $P$ and by the clique on $W$, while the edges inside $C$ and between $C$ and $W$ are lower order because $\delta^2=o(m)$ and $\delta Q=o(m)$.

The embeddings obtained by sending $d$ into $C$, sending $i_1,i_2$ into $P$, and sending $r_1,r_2,r_3$ into $W$ contribute
\[
|C||P|^2(|W|)_3=\Theta(\delta\Lambda^2Q^3)=\Theta(\delta^{5/2}\Lambda^{7/2})=\Theta(A_{F_\triangle}).
\]
Thus all three scales $\delta$, $\Lambda$, and $Q$ are active in the leading term. This example is a finite graph analogue of the three intervals in the threshold graphon calculation.

\subsection{Hall saturated bipartite examples}

Let $K_{s,t}$ be complete bipartite with $1\le s\le t$. The smaller colour class has a matching into the larger colour class, so $\Delta(K_{s,t})=0$ and $\sigma(K_{s,t})=s$. Theorem \ref{thmbipmain} gives $\Nbip(n,m,K_{s,t})=\Theta_{s,t}(m^sn^{t-s})$ whenever $m/n\to\infty$ and $m/n^2\to0$. In particular, $C_4=K_{2,2}$ satisfies $\Nbip(n,m,C_4)=\Theta(m^2)$ in this range.

More generally, every even cycle $C_{2r}$ has $\Delta(C_{2r})=0$ and $\sigma(C_{2r})=r$. Hence $\Nbip(n,m,C_{2r})=\Theta_r(m^r)$ in the same subquadratic range. These examples show that the quasi-complete bipartite graph has the correct order for natural Hall saturated families.

\subsection{Double stars and Hall deficiency}

Let $D_{a,b}$ be the double star whose adjacent centres are $u$ and $v$, with $a$ leaves adjacent to $u$ and $b$ leaves adjacent to $v$. Assume $2\le a\le b$. Let $L_u$ be the set of leaves adjacent to $u$, and let $L_v$ be the set of leaves adjacent to $v$. A smaller colour class is $A=\{v\}\cup L_u$, so $\sigma(D_{a,b})=a+1$. We have $\Delta(D_{a,b})=a-1$. Indeed, if $X\subseteq A$ and $v\notin X$, then $X\subseteq L_u$, so $|X|-|N(X)|\le a-1$, with equality when $X=L_u$. If $v\in X$, write $X=\{v\}\cup Y$ with $Y\subseteq L_u$. Then $N(X)=\{u\}\cup L_v$, so $|N(X)|=b+1$ and $|X|-|N(X)|=1+|Y|-(b+1)\le a-b\le0$. Hence the maximum deficiency is $a-1$.

The quasi-complete bipartite scale for $D_{a,b}$ is $m^{a+1}n^{b-a}$. The Hall deficient finite kernel sends all leaves to classes of scale $n$ and sends both centres to classes of scale $m/n$. This gives a labelled embedding count of order $m^2n^{a+b-2}$, and the ratio to the quasi-complete scale is of order $(n^2/m)^{a-1}$. Theorem \ref{thmbipmain} gives the explicit lower bound
\[
\frac{\Nbip(n,m,D_{a,b})}{N(B_n^m,D_{a,b})}\ge c_{a,b}\left(\frac{n^2}{m}\right)^{a-1}
\]
for all sufficiently large $n$ whenever $m/n\to\infty$ and $m/n^2\to0$. Since $a\ge2$, this ratio tends to infinity throughout the subquadratic range.

\subsection{Concluding remarks}

We end by summarizing the three finite-kernel conclusions proved in this paper and by proposing a sharp-constant problem left open by our results.

First, Theorem~\ref{thmds} proves the sparse threshold conjecture of Day and Sarkar. For every fixed graph $H$ without isolated vertices, it gives the sharp first-order asymptotic
\[
\cM_H(\beta)
=
\beta^{|V(H)|-\alpha^*(H)}(C_T(H)+o(1))
\]
as $\beta\to0$. The constant $C_T(H)$ is obtained from the finite state polynomial $P_H(q)$, and it is attained by the three-step threshold graphons $T_\beta(q)$. Thus, in the sparse graphon problem, the leading term is completely described by the admissible half-integral states of the fractional independence program.

Second, Theorem~\ref{thmbp}, proved through the stronger weighted embedding statement Theorem~\ref{thmweightedthreshold}, gives an affirmative answer to the sparse finite-graph question of Blekherman and Patel in the range $m\to\infty$ and $m=o(n^{3/2})$. Blekherman and Patel had proved the corresponding asymptotic thresholdization result in the supercritical range $m=\omega(n^{3/2})$, while their local-move method has a genuine loss at the transition scale $m=\Theta(n^{3/2})$, where the accumulated error is only $O(1)$ rather than $o(1)$. The finite-kernel method used here does not encounter this obstruction. Although Theorem~\ref{thmbp} is stated in the range asked about in~\cite{BlekhermanPatel2024}, the proof only uses the scale separations $\frac{\delta}{\Lambda}\to 0$, $ \frac{\delta}{Q}\to 0$, $\frac{Q}{\Lambda}\to 0$, 
where $\Lambda=\min\{n,m\}$, $\delta=m/\Lambda$, and $Q=\sqrt m$. These separations continue to hold whenever $m\to\infty$ and $m=o(n^2)$, so the same argument extends to the full subquadratic range after replacing Lemma~\ref{lemscaleseparation} by this strengthened form. The proof then proceeds by selecting a synchronized high-degree core, encoding the leading admissible states by common-neighbourhood data around the core, and realizing this data inside a chained threshold kernel.

Third, Theorem~\ref{thmbipmain} determines the correct order of magnitude for the bipartite host problem in terms of the scale $\kappa_H(n,m)$. This reduces the remaining sharp asymptotic problem to a finite-dimensional optimization problem. A natural goal is to determine the leading constant and the extremal finite kernels associated with the exposed faces of the logarithmic program defining $\kappa_H(n,m)$.

\begin{problem}\label{prob:bipconstants}
Let $H$ be a fixed bipartite graph without isolated vertices. Determine the sharp leading constant, and characterize the extremal finite kernels, in the asymptotic formula for $\Mbip(n,m,H)$ in the range $n\le m\le n^2/4$. More specifically, for fixed $p\in[1,2]$ and sequences $m=n^{p+o(1)}$ with $n\le m\le n^2/4$, determine the possible subsequential limits of $\frac{\Mbip(n,m,H)}{\kappa_H(n,m)}.$
\end{problem}

A natural formulation of Problem~\ref{prob:bipconstants} is to identify the finite-dimensional variational problems associated with the limiting exposed optimal faces of the logarithmic kernel program $\alpha_p(H)$. These variational problems should give the sharp constants and describe the extremal finite kernels.

\bibliographystyle{plain}
\bibliography{refs}
\end{document}